\documentclass[reqno,centertags, 11pt]{article}

\usepackage{amsmath,amsthm,amscd,amssymb}

\usepackage{latexsym}

\usepackage{mathrsfs}

\usepackage{bbold}

\usepackage{tikz,subfigure,graphics,overpic}

\usetikzlibrary{shapes,snakes}

\newcommand{\bbN}{{\mathbb{N}}}

\newcommand{\bbR}{{\mathbb{R}}}

\newcommand{\bbZ}{{\mathbb{Z}}}

\newcommand{\bbK}{{\mathbb{K}}}

\newcommand{\bbC}{{\mathbb{C}}}

\renewcommand{\Re}{\mathop{\mathrm{Re}}}

\DeclareMathOperator{\Var}{Var}

\newcommand{\no}{\nonumber}

\newcommand{\ti}{\tilde  }

\newcommand{\beq}{\begin{equation}}

\newcommand{\eeq}{\end{equation}}

\newcommand{\ba}{\begin{align}}

\newcommand{\ea}{\end{align}}

\newcommand{\loza}{ --++(-2,0)--++(0,2)--++(2,0)--++(0,-2)
[fill=gray!30!white] }

\newcommand{\lozb}{--++(-2,2)--++(0,-2)--++(2,-2)--++(0,2) 
[fill=gray!60!white]}

\newcommand{\lozc}{--++(-2,2)--++(-2,0)--++(2,-2)--++(2,0)
}

\allowdisplaybreaks

\numberwithin{equation}{section}

\newtheorem{theorem}{Theorem}[section]

\newtheorem{proposition}[theorem]{Proposition}

\newtheorem{lemma}[theorem]{Lemma}

\theoremstyle{definition}

\newtheorem{definition}[theorem]{Definition}

\newtheorem{Remark}[theorem]{Remark}

\newenvironment{remark*}{\begin{Remark}\rm} { \end{Remark}}

\newtheorem{remark}[theorem]{Remark}

\title{Nonintersecting paths with a staircase initial condition}
\author{Jonathan Breuer\footnote{ Einstein Institute of Mathematics, The Hebrew University of Jerusalem,
Jerusalem 91904, Israel. E-mail: jbreuer@math.huji.ac.il. Supported in part by The Israel Science Foundation (grant no.\ 1105/10)} \and 
Maurice Duits\footnote{Department of Mathematics 253-37, California Institute of Technology, 1200 E. California Blvd, CA 91125. Current address: Department of Mathematics, Royal Institute of Technology (KTH), Lindstedtsv\"agen 25, SE-10044, Stockholm, Sweden. Email: duits@kth.se. Supported in part by the grant KAW 2010.0063 from the
Knut and Alice Wallenberg Foundation}}
\date{}
\begin{document}
\maketitle
\sloppy
\begin{abstract}
We consider an ensemble of $N$ discrete nonintersecting paths starting from equidistant points and ending at consecutive integers. Our first result is an explicit formula for the correlation kernel that allows us to analyze the process as $N\to \infty$.   In that limit we obtain a new general class of kernels describing the local correlations close to the equidistant starting points.
As the distance between the starting points goes to infinity, the correlation kernel converges to that of a single random walker.
As the distance to the starting line increases, however, the local correlations converge to the sine kernel.
Thus, this class interpolates between the sine kernel and an ensemble of independent particles.
We also compute the scaled simultaneous limit, with both the distance between particles and the distance to the starting line going to infinity, and obtain a process with number variance saturation, previously studied by Johansson.
\end{abstract}


\section{Introduction}


Nonintersecting path ensembles are a natural arena for generating and studying a variety of mathematical and physical phenomena. At the same time, they often have a nice integrable structure that allows for explicit computations. As such they constitute an important source of integrable systems in random matrix theory.  In the discrete context, for example, nonintersecting paths serve as an important link between representation theory, random matrix theory and combinatorics (see, e.g., \cite{BKMM,Forrester,J2,J} and the references therein for excellent discussions on the topic).

Typically, an ensemble of nonintersecting paths consists of a number of random walkers with prescribed initial and final positions that do not collide while performing the walk. 
The interaction between the walkers depends strongly on the initial and final positions of the walkers. For example, if these positions are close to each other (e.g., they start and end at the same point in the continuous case or consecutive points on a grid in the discrete case) the condition that they can never collide results in a strong repulsive interaction. However, if the initial and final positions are far apart, the effect of the forbidden intersection is much weaker, as the probability of colliding is small. 

In the present paper, we discuss a discrete model of nonintersecting paths  that interpolates between these two situations. We give a precise definition later on (see also Figure \ref{fig:paths}). Here we content ourselves with describing the central features. The most important feature is that  the initial positions are equally spaced with gap $k\in \bbN$ (instead of densely packed). The final positions are densely packed.  If $k$ is large, the interaction between the walkers right after the initial positions is small. On the other hand, if $k$ is small  the nonintersecting condition leads to repulsion that is substantially felt from the start.

We shall show that eventually, as we move  away from the starting points,  the local correlations converge to the usual sine kernel limit. However, this takes some time, and our main interest is in describing the process before universality is reached. In this regime, we obtain a new family of processes that depends on the parameter $k$ (and also on the parameters from the underlying random walks). Each member of the new family is a discrete determinantal point process. As we shall see, this family admits several interesting limits. In one limit we move away from the initial positions, arriving in the bulk regime, and  obtain the discrete sine kernel in the limit (i.e., the universality limit). In fact, we are able to generate a subclass (not all) of the extensions of the discrete sine kernel as introduced in \cite{Bor}. In a different limit, $k\to \infty$, we show that the process turns into a single random walker. In this sense, the family interpolates between two fundamental objects: the single random walk and the sine process.

For $k=1$ the model was discussed by Johansson in \cite{J2} to give a nonintersecting path interpretation for the Schur measure introduced in \cite{Ok1}. Thus, the family of processes we consider can be seen as a natural generalization of the Schur measure. Moreover, Johansson's model served as important inspiration for the definition of the Schur process by Okounkov and Reshetikhin \cite{OR}. In \cite{ImSa}, the authors introduced a generalization of the Schur process by allowing the particles to have arbitrary initial positions. Along the way, we will also prove a result (see Theorem \ref{generalFormula}) on the determinantal structure for the case of more general final and initial locations. 

A particular specialization of the Schur process (interpreting the time parameter as an extra spatial dimension) is the $q^{volume}$ weight for boxed plane partitions. We will show that this picture generalizes to general $k$ (in fact, to general initial spacings $\{k_j\}$), in the sense that a specialization of the process leads to a  $q^{volume}$ weight on certain three dimensional box configurations. However, whereas boxed plane partitions may be viewed as boxes placed in a corner, a similar interpretation for our model leads to boxes placed on a staircase. This is the reason for the name of the model.

It is interesting to note that Johansson's original model was described as a growth model \cite{J2}. Such an interpretation exists for the model described here as well, where in the general case $k>1$, the growth model obtained may be seen as a growth model with a certain \emph{random} ``staircase shaped'' initial condition. While it is not our intention to pursue this line of reasoning further  in this paper,  we will  further explore these connections in a subsequent publication.

It is also interesting to consider the continuum limit which places our model in the framework of Dyson Brownian motion \cite{dyson,grabiner} or GUE with external source \cite{BH}. This continuous analogue  played an important role in Johansson's proof of universality for the local correlation for GUE divisible Wigner matrices \cite{J4}.  A central question in this context regards the time it takes for the local correlations to `relax' to  the usual universality class (see \cite{Erd} and references therein). For the discrete model analyzed in this paper, we consider a significantly shorter time scale. Namely, we consider the local correlations at the very first steps after the starting points.

Finally, motivated by the study of the zeros of the Riemann zeta function, Johansson, \cite{J3}, used a continuous model to construct a family of determinantal point processes with number variance saturation. We will show that the process introduced in  \cite{J3} can be obtained as a special limit of our family. Moreover, we shall present a soft argument to show that number variance saturation holds in our setting as well.

The rest of this paper is structured as follows. The model is described in Sections 2.1 and  2.2. Our results are described in Sections 2.3--2.6. The proofs are given in Section 3.  
\subsubsection*{Acknowledgments}
We thank Alexei Borodin for many useful discussions. M.D.\ acknowledges the hospitality of the Einstein Institute of Mathematics at the Hebrew University of Jerusalem, where some of this work was done.

\section{Statement of the results}

\subsection{Nonintersecting paths}

Let us now introduce the model. See also Figure \ref{fig:paths}. Fix $N \in \bbN$ and let $|\alpha_j| <1$, $|\beta_j| <1$, ($j=1,\ldots,N$) be fixed parameters. Let, further, $k_1<k_2<k_3<\ldots<k_N$ and $l_1<l_2<l_3<\ldots<l_N$ be finite increasing sequences of integers. By shifting the model we may assume that $k_1=0$ and we do so from now on. Consider $N$ particles (=\lq walkers\rq) initially positioned on the vertical line $s=-N$ at heights $x_j=k_j$, $j=1\ldots N$.  
The particles perform walks with $2N$ steps each (plus a `semistep' at the end), ending on the line $s=N$ at the heights $x_j=l_j$.
The individual walks are described as follows: each step involves a horizontal jump and a random vertical jump. All horizontal jumps, except for the first step and final semistep, are jumps of size $1$ to the right. 
The first and last horizontal jumps are jumps of size $1/2$ to the right.
The vertical jumps occur on the lines $s \in \bbZ+1/2$. To the left of the middle line, $s=0$, all vertical jumps are upwards, and to the right of it, all vertical jumps are downwards.

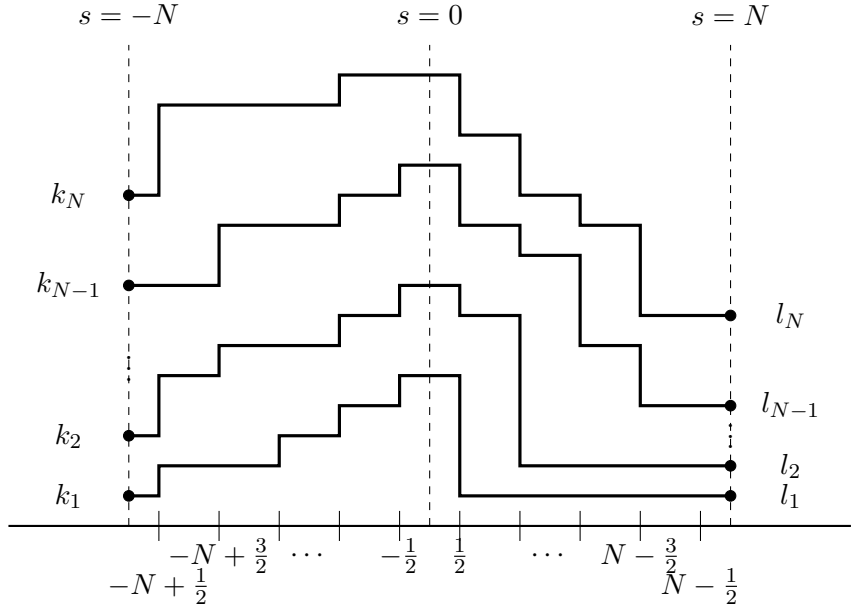
\begin{figure}[t]
\begin{center}
\begin{tikzpicture}[xscale=0.4,yscale=0.2]
\filldraw[yscale=2] (-10,0) circle(5pt);
\filldraw[yscale=2] (-10,2) circle(5pt);
\draw (-10,9) node{$\vdots$};
\filldraw[yscale=2] (-10,7) circle(5pt);
\filldraw[yscale=2] (-10,10) circle(5pt);
\filldraw (-12,0) node {$k_1$};
\filldraw (-12,4) node {$k_2$};
\filldraw (-12,14) node {$k_{N-1}$};
\filldraw (-12,20) node {$k_{N}$};

\filldraw [yscale=2](10,0) circle(5pt);
\filldraw[yscale=2] (10,1) circle(5pt);
\draw (10,4.5) node{$\vdots$};
\filldraw[yscale=2] (10,3) circle(5pt);
\filldraw[yscale=2] (10,6) circle(5pt);
\filldraw (12,0) node {$l_1$};
\filldraw (12,2) node {$l_2$};
\filldraw (12,6) node {$l_{N-1}$};
\filldraw (12,12) node {$l_N$};

\draw[very thick] (-10,0)--(-9,0)--(-9,2)--(-5,2)--(-5,4)--(-3,4)--(-3,6)--(-1,6)--(-1,8)--(1,8)--(1,0)--(10,0);
\draw[very thick] (-10,4)--(-9,4)--(-9,8)--(-7,8)--(-7,10)--(-3,10)--(-3,12)--(-1,12)--(-1,14)--(1,14)--(1,12)--(3,12)--(3,2)--(10,2);
\draw[very thick] (-10,14)--(-9,14)--(-9,14)--(-7,14)--(-7,18)--(-3,18)--(-3,20)--(-1,20)--(-1,22)--(1,22)--(1,18)--(3,18)--(3,16)--(5,16)--(5,10)--(7,10)--(7,6)--(10,6);
\draw[very thick] (-10,20)--(-9,20)--(-9,26)--(-3,26)--(-3,28)--(-1,28)--(-1,28)--(1,28)--(1,24)--(3,24)--(3,20)--(5,20)--(5,18)--(7,18)--(7,12)--(10,12);

\draw[dashed] (0,-2)--(0,30);
\draw (0,32) node{$s=0$};
\draw[dashed] (-10,-2)--(-10,30);
\draw (-10,32) node{$s=-N$};
\draw[dashed] (10,-2)--(10,30);
\draw (10,32) node{$s=N$};

\draw[thick] (-14,-2)--(14,-2);
\foreach \y in {-9,-7,...,9}
     		\draw[-] (\y,-3) -- (\y,-1); 

\draw (-9,-6) node{$-N+\tfrac{1}{2}$};
\draw (-7,-4) node{$-N+\tfrac{3}{2}$};
\draw (7,-4) node{$N-\tfrac{3}{2}$};
\draw (9,-6) node{$N-\tfrac{1}{2}$};
\draw (1,-4) node{$\tfrac{1}{2}$};
\draw (-1,-4) node{$-\tfrac{1}{2}$};
\draw (4,-4) node{$\cdots$};
\draw (-4,-4) node{$\cdots$};
\end{tikzpicture}
\end{center}
\caption{The model of nonintersecting paths. We will be mainly interested in the equally spaced situation $k_j=k(j-1)$ for $j=1,\ldots,N$ and some $k\in \bbN$ and $l_j=j-1$ for $j=1,\ldots,N$.}
\label{fig:paths}
\end{figure}
The first horizontal jump is of size $1/2$, to the line $s=-N+1/2$. The first vertical jump takes place on the line $s=-N+1/2$ and is an upwards jump of a random magnitude, where a jump of size $m$ is assigned the weight $\alpha_1^m$ (geometric distribution). 
The next horizontal jump is of size $1$ (to the line $s=-N+3/2$) where the particle now performs a vertical jump upwards of size $m$ with weight $\alpha_2^m$. This continues up to the middle line, $s=0$.
As noted earlier, the vertical jumps for each of the first $N$ steps are all upwards and the weight of a jump of size $m$ at the $n$'th step ($1 \leq n \leq N$) 
is given by $\alpha_n^m$. Thus, when all the particles have intersected the middle line, $s=0$, they have all traced random `up-right' paths. 
After jumping horizontally from the line $s= -1/2$ to the line $s=1/2$, the vertical jumps become downward jumps and the weight of a jump of size $m$ at the $n$'th step 
($N+1 \leq n \leq 2N$) is $\beta_{2N-n+1}^{m}$. The last ($2N$'th) vertical jump occurs on the line $s= N-1/2$. The last semistep is simply a horizontal jump of size $1/2$ from $s=N-1/2$ to the line $s=N$. 
Finally, we condition on the paths never to intersect. 

We want to describe the point configuration obtained by considering the intersections of the paths with the vertical lines $s=-(N-1),-(N-2), \ldots, 0$. We shall use the theory of determinantal point processes in Section 2.3 to do this, and we shall be mainly interested in the equally spaced situation: 
\begin{center}
{$k_j=k(j-1)$ for some $k \in \bbN$ and $l_j=j-1$.}
\end{center} First, however, we  present a different way of viewing this model. 

\subsection{An equivalent tiling model}
Our model of nonintersecting paths is equivalent to  lozenge tilings of a particular domain.  As mentioned in the Introduction, for $k_j=l_j=j-1$ this equivalence can already be found in \cite{OR} in connection to  a particular instance of the Schur process. Here we discuss the natural generalization to general $\{k_j\}_{j=1}^N$ (keeping $l_j=j-1$). 

The domains that we are tiling are defined in the following way (see also Figure \ref{fig:tiledomain} for an illustration). We start with the semi-infinite rectangle $[-N,N] \times \left(-\tfrac{1}{2},\infty\right)$. On the left vertical line we cut out $N$ right triangles with legs of unit length that are parallel to the axes and such that the points $(-N,k_j)$ are at the centers of the vertical legs. We now consider all possible tilings of this domain with the lozenges \begin{center}
\tikz[scale=0.20] \draw (-2,0) \loza;\, , \  \tikz[scale=0.20] \draw (0,0) \lozb;\, and  \tikz[scale=0.20] \draw (0,0) \lozc;  
\end{center}
The vertices of the lozenges are placed on the grid $\{-N,-N+1,\ldots,N\}\times \left(\bbN-\tfrac{1}{2}\right)$.

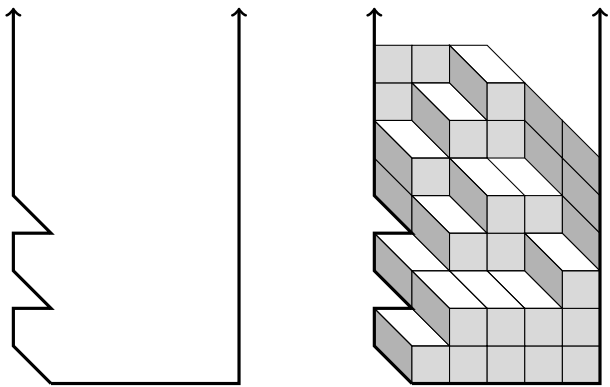
\begin{figure}[t]
\begin{center}
\begin{tikzpicture}[scale=0.25]
\draw[very thick,->] (-4,-2)--(-6,0)--(-6,2)--(-4,2)--(-6,4)--(-6,6)--(-4,6)--(-6,8)--(-6,18) ;
\draw[very thick, ->] (-4,-2)--(6,-2)--(6,18); \end{tikzpicture}
\hspace*{1.5cm}
\begin{tikzpicture}[scale=0.25]
 \draw (-4,1) \lozb ;
 \draw (-4,5) \lozb; 
  \draw (-4,9)\lozb; \draw (-4,11)
  \lozb; \draw (-4,13) \loza; \draw (-4,15) \loza;

\draw (-2,-1) \loza;
\draw (-2,1) \lozc; \draw (-2,3) \lozb; \draw (-2,5) \lozc; \draw (-2,7) \lozb; \draw (-2,9) \loza; \draw (-2,11) \lozc; \draw (-2,13) \lozb; \draw (-2,15) \loza;

\draw (-0,-1) \loza; 
\draw (0,1) \loza; \draw (0,3) \lozc; \draw (0,5) \loza; \draw (0,7) \lozc; \draw (0,9) \lozb; \draw (0,11) \loza; \draw (0,13) \lozc; \draw (0,15) \lozb;

\draw (2,-1) \loza; 
\draw (2,1) \loza; \draw (2,3) \lozc; \draw (2,5) \loza; \draw (2,7) \loza; \draw (2,9) \lozc; \draw (2,11) \loza; \draw (2,13) \loza; \draw (2,15) \lozc;

\draw (4,-1) \loza; 
\draw (4,1) \loza; \draw (4,3) \lozc; \draw (4,5) \lozb; \draw (4,7) \loza; \draw (4,9) ; \draw (4,11) \lozb; \draw (4,13) \lozb;   

\draw (6,-1) \loza; 
\draw (6,1) \loza; \draw (6,3) \loza; \draw (6,5) \lozc; \draw (6,7) \lozb; \draw (6,9) \lozb; \draw (6,11) \lozb;  

\draw[very thick,->] (-4,-1)--(-6,1)--(-6,3)--(-4,3)--(-6,5)--(-6,7)--(-4,7)--(-6,9)--(-6,19) ;
\draw[very thick, ->] (-4,-1)--(6,-1)--(6,19);
\end{tikzpicture}

\end{center}
\caption{The left picture shows the domain that we tile for the special case $N=3$ and $k_j=2j$. The right picture shows a particular tiling of this domain.}
\label{fig:tiledomain}
\end{figure}

Given $\{k_j\}_{j=1}^N$, we claim that there is a one-to-one correspondence between lozenge tiling of such domains and the nonintersecting path model introduced above.  To see this, draw a collection of nonintersecting paths as in Figure \ref{fig:sheartransform}. In the right half plane we perform a shear transformation $(s,x)\mapsto (s,x-s)$. Note that in Figure \ref{fig:sheartransform} the parts of the paths that end up in the lower half plane are slanted straight lines. In fact, by the nonintersecting condition and by the fact that the final points are consecutive integers, this holds for any possible configuration of paths. Hence these parts can be discarded.

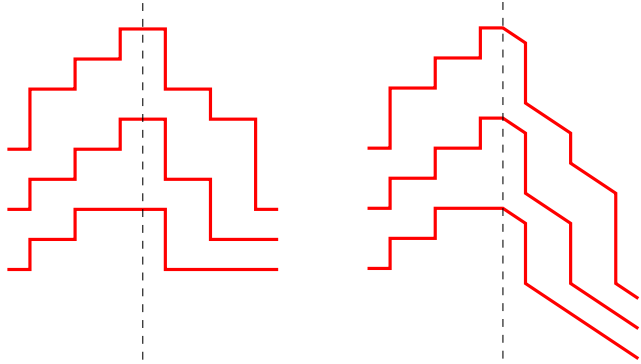
\begin{figure}[t]
\begin{center}
\subfigure{\begin{tikzpicture}[xscale=0.3,yscale=0.2]
\draw[very thick,color=red] (-6,8)--(-5,8) -- (-5,12)--(-3,12)--(-3,14)  --(-1,14)--(-1,16) --(1,16)--(1,12)--(3,12)--(3,10)--(5,10)--(5,4)--(6,4);
\draw[very thick,color=red] (-6,4)--(-5,4)  --(-5,6) -- (-3,6)--(-3,8) -- (-1,8)--(-1,10)--(1,10) --(1,6)--(3,6)--(3,2)--(6,2) ;
\draw[very thick,color=red] (-6,0)--(-5,0)--(-5,2) --(-3,2)--(-3,4) -- (-1,4)--(-1,4)--(1,4) -- (1,0) --(6,0);
\draw[dashed] (0,-6)--(0,18);
\end{tikzpicture}
}
\hspace{.05 \textwidth}
\subfigure{
\begin{tikzpicture}[xscale=0.3,yscale=0.2]
\draw[very thick,color=red] (-6,8)--(-5,8) -- (-5,12)--(-3,12)--(-3,14)  --(-1,14)--(-1,16) --(0,16)--(1,15)--(1,11)--(3,9)--(3,7)--(5,5)--(5,-1)--(6,-2);
\draw[very thick,color=red] (-6,4)--(-5,4)  --(-5,6) -- (-3,6)--(-3,8) -- (-1,8)--(-1,10)--(0,10)--(1,9) --(1,5)--(3,3)--(3,-1)--(6,-4) ;
\draw[very thick,color=red] (-6,0)--(-5,0)--(-5,2) --(-3,2)--(-3,4) -- (-1,4)--(-1,4)--(0,4)--(1,3) -- (1,-1) --(6,-6);
\draw[dashed] (0,-6)--(0,18);
\end{tikzpicture}
}
\end{center}
\caption{The left picture shows a particular path configuration for $N=3$ and $k_j=2j$. In the right picture we show the result of the shear transformation $(s,x)\mapsto(s,x-s)$ in the right half.}
\label{fig:sheartransform} 
\end{figure}

Draw the paths (minus the trivial parts) and the domain together in one  picture as shown in Figure \ref{fig:tiledomainpaths}. Starting from this configuration of paths, we tile the domain according to the rules in Figure \ref{fig:tilerules}.   We have different rules in the left half plane and in the right half plane. In the left half plane the horizontal parts of the paths are associated with  \tikz[scale=0.10]  \draw (0,0) \lozc;,\  the vertical parts with   \tikz[scale=0.10]  \draw (0,0) \lozb;  and the third type is associated with the blank spaces.   In the right half plane the slanted parts of the paths are associated with  \tikz[scale=0.10]  \draw (0,0) \lozc;, the vertical parts with   \tikz[scale=0.10]  \draw (0,0) \loza;  and the third type is associated with the blank spaces.  In order to get a consistent tiling we need an extra rule on the middle line $s=0$.

It is not difficult to see that this indeed always induces a lozenge tiling. Moreover, given any  lozenge tiling  we can construct a unique path configuration that leads to this tiling according to the rules in Figure \ref{fig:tilerules}.  This means that we have established a  one-to-one correspondence between our model of nonintersecting paths and lozenge tiling of domains as in Figure \ref{fig:tiledomain}. Note that in the special case $k_j=j-1$  (the densely packed case), the tiling in the lower left corner of the domain is uniquely determined by the boundary and so the boundary there can be modified to a line segment. In fact, in that case the problem is reduced to tiling a hexagon for which the vertical sides have infinite length.

\begin{figure}[t]
\begin{center}
\subfigure[Tiling rules  in the left half]{
\begin{tikzpicture}[scale=0.3]
\draw[thick] (2,0) \lozc;
\draw[very thick,color=red] (-1,1)--(1,1);
\draw[thick] (-6,0)\loza;
\draw[thick] (-2,-2) \lozb;
\draw[very thick,color=red] (-3,-1)--(-3,-3);
\end{tikzpicture}
}\hspace*{1cm}
\subfigure[Tiling rules in the right half]{

\begin{tikzpicture}[scale=0.3]
\draw[thick] (2,0) \lozc;
\draw[very thick,color=red] (-1,2)--(1,0);
\draw[thick] (-6,0)\loza;
\draw[very thick,color=red] (-7,2)--(-7,0);
\draw[thick] (-2,-2) \lozb;
\end{tikzpicture}
}
\hspace*{1cm}\subfigure[Tiling rule on the middle line]{
\begin{tikzpicture}[scale=0.6]
\draw[very thick,color=red] (-3,1)--(-2.5,1)--(-2,0.5);
\draw[thick] (-2.5,0.5)--(-3.5,1.5)--(-2.5,1.5)--(-1.5,0.5)--(-2.5,0.5);
\draw[color=white] (-4,-0.5)--(-1,-0.5);
\end{tikzpicture}
}
\end{center}
\caption{The different tiling rules.}
\label{fig:tilerules} \end{figure}
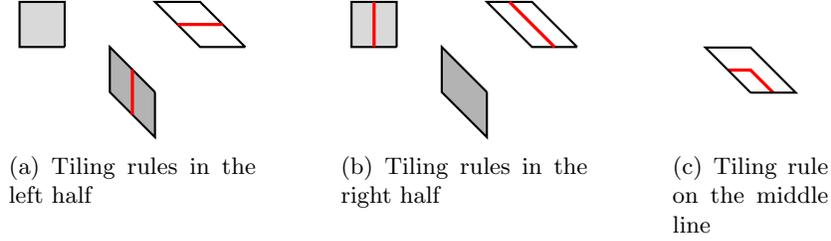
\begin{figure}[t]
\begin{center}

\begin{tikzpicture}[scale=0.25]
\draw[very thick,color=red] (-6,8)--(-5,8) -- (-5,12)--(-3,12)--(-3,14)  --(-1,14)--(-1,16) --(0,16)--(1,15)--(1,11)--(3,9)--(3,7)--(5,5)--(5,-1);
\draw[very thick,color=red] (-6,4)--(-5,4)  --(-5,6) -- (-3,6)--(-3,8) -- (-1,8)--(-1,10)--(0,10)--(1,9) --(1,5)--(3,3)--(3,-1);
\draw[very thick,color=red] (-6,0)--(-5,0)--(-5,2) --(-3,2)--(-3,4) -- (-1,4)--(-1,4)--(0,4)--(1,3) -- (1,-1); 
\draw[very thick,->] (-4,-1)--(-6,1)--(-6,3)--(-4,3)--(-6,5)--(-6,7)--(-4,7)--(-6,9)--(-6,19) ;
\draw[very thick, ->] (-4,-1)--(6,-1)--(6,19);
\end{tikzpicture}
\hspace*{1.5cm}
\begin{tikzpicture}[scale=0.25]

 \draw (-4,1) \lozb ;
 \draw (-4,5) \lozb; 
  \draw (-4,9)\lozb; \draw (-4,11)
  \lozb; \draw (-4,13) \loza; \draw (-4,15) \loza;

\draw (-2,-1) \loza;
\draw (-2,1) \lozc; \draw (-2,3) \lozb; \draw (-2,5) \lozc; \draw (-2,7) \lozb; \draw (-2,9) \loza; \draw (-2,11) \lozc; \draw (-2,13) \lozb; \draw (-2,15) \loza;

\draw (-0,-1) \loza; 
\draw (0,1) \loza; \draw (0,3) \lozc; \draw (0,5) \loza; \draw (0,7) \lozc; \draw (0,9) \lozb; \draw (0,11) \loza; \draw (0,13) \lozc; \draw (0,15) \lozb;

\draw (2,-1) \loza; 
\draw (2,1) \loza; \draw (2,3) \lozc; \draw (2,5) \loza; \draw (2,7) \loza; \draw (2,9) \lozc; \draw (2,11) \loza; \draw (2,13) \loza; \draw (2,15) \lozc;

\draw (4,-1) \loza; 
\draw (4,1) \loza; \draw (4,3) \lozc; \draw (4,5) \lozb; \draw (4,7) \loza; \draw (4,9) ; \draw (4,11) \lozb; \draw (4,13) \lozb;   

\draw (6,-1) \loza; 
\draw (6,1) \loza; \draw (6,3) \loza; \draw (6,5) \lozc; \draw (6,7) \lozb; \draw (6,9) \lozb; \draw (6,11) \lozb;  


 
 \draw[very thick,color=red] (-6,8)--(-5,8) -- (-5,12)--(-3,12)--(-3,14)  --(-1,14)--(-1,16) --(0,16)--(1,15)--(1,11)--(3,9)--(3,7)--(5,5)--(5,-1);
\draw[very thick,color=red] (-6,4)--(-5,4)  --(-5,6) -- (-3,6)--(-3,8) -- (-1,8)--(-1,10)--(0,10)--(1,9) --(1,5)--(3,3)--(3,-1);
\draw[very thick,color=red] (-6,0)--(-5,0)--(-5,2) --(-3,2)--(-3,4) -- (-1,4)--(-1,4)--(0,4)--(1,3) -- (1,-1); 
\draw[very thick,->] (-4,-1)--(-6,1)--(-6,3)--(-4,3)--(-6,5)--(-6,7)--(-4,7)--(-6,9)--(-6,19) ;
\draw[very thick, ->] (-4,-1)--(6,-1)--(6,19);
\end{tikzpicture}
\
\end{center}
\caption{The left picture shows the domain and a particular configuration of paths. The right picture shows the corresponding lozenge tiling according to the rules in Figure \ref{fig:tilerules}.}
\label{fig:tiledomainpaths}
\end{figure}
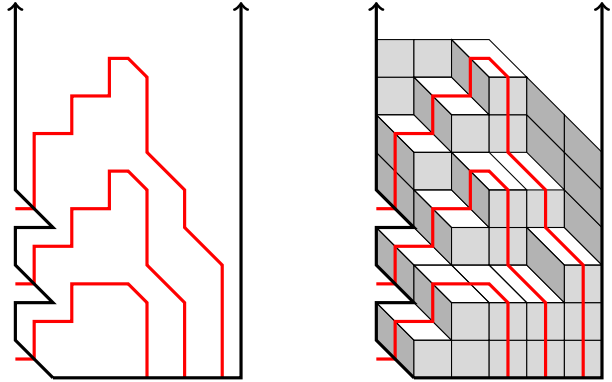

Lozenge tilings can also be viewed as a model in three dimensions. Indeed, a tiling is a two-dimensional depiction of a stack of boxes (see Figure \ref{fig:tiledomain}). For $k_j={j-1}$ each tiling is a boxed plane partition. In that case, the boxes are placed in a corner of a room with flat walls and floor. The weight that we put on paths leads to a natural weight on the lozenge tilings and hence on boxed plane partitions. Without proof we claim that the weight of the tiling $\mathcal T$ is proportional to
\beq \label{eq:tilingweight}
w(\mathcal T)=\prod_{\tikz[scale=0.10] \draw(0,0) \lozc;} w\left(\tikz[scale=0.10] \draw(0,0) \lozc;\right),
\eeq
where
\beq
w\left(\tikz[scale=0.10] \draw(0,0) \lozc; \right)
=\begin{cases} \left(\frac{\alpha_{N+i}}{\alpha_{N+i+1}}\right)^j, & i=-N+1,\dots,-1\\
\left(\alpha_N\beta_N\right)^j, & i=0\\
\left(\frac{\beta_{N-i}}{\beta_{N-i+1}}\right)^j, & i=1,\dots,N-1.\\
\end{cases}
\eeq
and $(i,j)$ are the coordinates of the centers of tiles \tikz[scale=0.10] \draw(0,0) \lozc; in $\mathcal T$. 
Note that if we set $\alpha_i=\beta_i=q^{\frac{1}{2}+N-i}$ for some $q\in (0,1)$ then the resulting weight falls in the class of $q^{volume}$ models. See \cite{OR} for the case $k_j=j-1$. For a construction of $q$-distribution on lozenge tilings and an extensive list of references see \cite{BGR}.


Note that for the special $k_j=j-1$, the room in which we place boxes, (the `initial condition'), can be constructed by drawing the tiling for the trivial configuration of paths, namely that of straight lines. In the general situation, a configuration of straight lines is not allowed. Nevertheless, there is a natural substitute for the trivial configuration. Consider the configuration for which the paths in the left half are straight lines, but after passing the vertical line $s=0$, at each step the walkers jump as much as the nonintersecting condition allows them. This leads to the staircase shaped paths as shown in Figure \ref{fig:staircase}.  In fact, if we draw the corresponding tiling (also shown in Figure \ref{fig:staircase}) then the tiling corresponds to part of a staircase in three dimensions. Note that this configuration  of paths or tiling is indeed the most basic one, in the sense that any other configuration can be constructed by placing boxes on top of this staircase.  Again, by the discussion around \eqref{eq:tilingweight}, for an appropriate choice of parameters our model falls in the class of $q^{volume}$ models, but now everything takes place on a staircase.

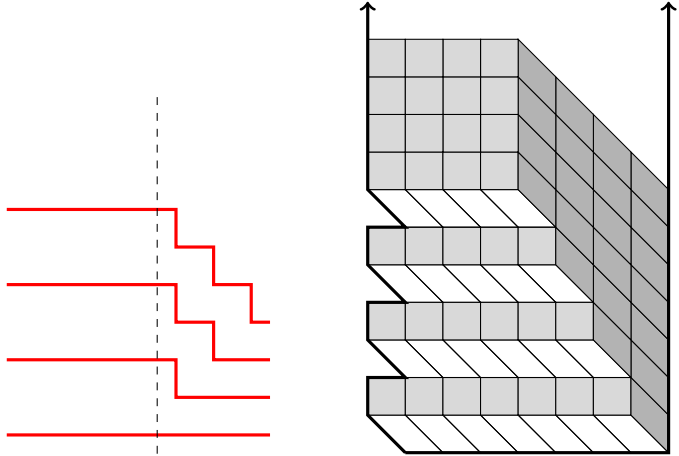
\begin{figure}[t]
\begin{center}
\begin{tikzpicture}[scale=0.25]
\draw[very thick,color=red] (-7,13)--(2,13)--(2,11)--(4,11)--(4,9)--(6,9)--(6,7)--(7,7);
\draw[very thick,color=red] (-7,9)--(2,9)--(2,7)--(4,7)--(4,5)--(7,5);
\draw[very thick,color=red] (-7,5)--(2,5)--(2,3)--(7,3);
\draw[very thick,color=red] (-7,1)--(7,1);
\draw[dashed] (1,0)--(1,19);
\end{tikzpicture}
\hspace*{0.5cm}
\hspace*{0.5cm}\begin{tikzpicture}[scale=0.25]

\foreach \y in {0,1,2,...,6} \draw (2*\y-6,-5) \lozc;
\foreach \y in {-1,0,1,2,...,5} \draw (2*\y-6,-3) \loza;
	\foreach \y in {0,1,2,...,5} \draw (2*\y-6,-1) \lozc;
	\foreach \y in {0,1,2,...,4} \draw (2*\y-6,3) \lozc;
	\foreach \y in {-1,0,1,2,...,4} \draw (2*\y-6,1) \loza;

	\foreach \y in {-1,0,1,2,...,3} \draw (2*\y-6,5) \loza;

	\foreach \y in {0,1,2,...,3} \draw (2*\y-6,7) \lozc;
	\foreach \y in {-1,0,1,2} \draw (2*\y-6,9) \loza;
	
		\foreach \y in {-1,0,1,...,5} \draw (6,-1+2*\y) \lozb;

		\foreach \y in {0,1,...,5} \draw (4,1+2*\y) \lozb;
		\foreach \y in {1,2,...,5} \draw (2,3+2*\y) \lozb;
		\foreach \y in {2,3,...,5} \draw (0,5+2*\y) \lozb;
		\foreach \y in {3,...,5} \draw (-2,5+2*\y) \loza;
		\foreach \y in {3,...,5} \draw (-8,5+2*\y) \loza;

		\foreach \y in {3,...,5} \draw (-4,5+2*\y) \loza;
		\foreach \y in {3,...,5} \draw (-6,5+2*\y) \loza;
		\draw[very thick,->] (-8,-5)--(-10,-3)--(-10,-1)--(-8,-1)--(-10,1)--(-10,3)--(-8,3)--(-10,5)--(-10,7)--(-8,7)--(-10,9)--(-10,19) ;
\draw[very thick, ->] (-8,-5)--(6,-5)--(6,19); 
\end{tikzpicture}
\end{center}
\caption{The \lq simplest\rq configuration of paths and the corresponding tiling. Viewed in three dimension the left picture is part of a staircase.}
\label{fig:staircase}
\end{figure}

\subsection{The process for finite $N$}

Returning to the path ensemble picture, we want to describe the resulting point process. We shall do this using the language and machinery of determinantal point processes.

A determinantal point process on a discrete set $X$ is a probability measure on  $2^X$  such that 
\beq\label{eq:detprocess}
\mathrm{Prob} (x_1,\ldots,x_n)=\det \left(K(x_i,x_j)\right)_{i,j=1}^n.
\eeq
for some kernel $K:X\times X\to\bbR$.  A determinantal point process is completely determined by the kernel $K$ (which is however not unique). For more details and background on determinantal point processes we refer to \cite{BorDet,HKPV,J,K,L,Sosh,Sosh2}.

We define  a point process on $\bbN \times \bbN$  by the locations of the intersections of the paths with the lines $s=-N,-N+1,\ldots,N-1,N$. More precisely, we define the process by
\beq
\begin{split}
\mathrm{Prob} ((s_1,x_1), & \ldots, (s_n,x_n))\\
&= \mathrm{Prob}\left (\textrm{paths pass through points } (s_j-N,x_j)\right).
\end{split}
\eeq
The shift $s_j-N$ in the right-hand side is introduced to simplify the upcoming formulas.

By a well known result of Lindstr\"om-Gessel-Viennot  \cite{J,Stem} the process is determinantal. Moreover, a formula for the kernel of this process is given by the Eynard-Mehta Theorem (see, e.g., \cite{BorDet}). However, the expression obtained is, in general, not useful for asymptotic analysis. In large part, this is due to the fact that this formula involves the inverse of a large Gramm matrix. Our first task is therefore to give a more explicit expression for the kernel.  


We will need some notation. For any $N$-tuple of complex numbers, $\gamma=(\gamma_1, \gamma_2, \ldots \gamma_N)$, let 
\beq \label{alternating}
h_{k_1,k_2,\ldots,k_N}(\gamma)= \det \left(\gamma_j^{k_i} \right)_{i,j=1}^N
\eeq
be the alternating polynomial of degree $k_1,k_2,\ldots,k_N$ in $\gamma_1,\gamma_2,\ldots \gamma_N$. 
Let further $(\beta; \beta_j \mapsto z)=(\beta_1, \beta_2, \ldots, \hat{\beta_j} \ z ,\ldots, \beta_N)$ be the $N$-tuple of $\beta$'s with $\beta_j$ replaced by $z$, and let
\beq \label{alternating-z}
h_{k_1,k_2,\ldots,k_N}(\beta; \beta_j \mapsto z)
\eeq
be the alternating polynomial in the $\beta$'s with $\beta_j$ replaced by $z$. Similarly, we define 
\beq 
h_{l_1,l_2,\ldots,l_N}(1/\beta;1/\beta_j \mapsto 1/z),
\eeq 
where we used the notation $1/\beta=(1/\beta_1,\ldots,1/\beta_N)$. We will also define \beq \label{Fell0}
F_\ell(z)=\left \{ \begin{array}{cc}  &(1-\alpha_\ell z)^{-1} \quad 1 \leq \ell \leq N \\
&\left(1-\frac{\beta_{2N-\ell+1}}{z} \right)^{-1} \quad N+1 \leq \ell \leq 2N  \end{array} \right. .
\eeq
which are  the generating functions for  the transition probabilities.

Our first result is an explicit expression for the kernel. 

\begin{theorem} \label{generalFormula}
Assume $\beta_j \neq \beta_\ell$ for $j \neq \ell$.
Let $0=k_1 <k_2 <\ldots <k_N$, $l_1<l_2<\ldots <l_N$ with $$l_N-k_1=l_N \leq N-1.$$ Then for any two points $(s_1, x_1), (s_2,x_2) \in \bbN^2$,
the kernel $K(s_1,x_1;s_2,x_2)$ is given by
\beq \label{kernel0}
\begin{split}
&K(s_1, x_1; s_2, x_2) =-{\mathbb 1}_{s_1>s_2} \frac{1}{2 \pi i} \oint_{\partial \mathbb{D}} \prod_{j=s_2+1}^{s_1} F_j(z) \frac{dz}{z^{x_1-x_2+1}} \\
& \quad +\sum_{j=1}^N  \frac{\prod_{r=1}^N \left(1-\alpha_r \beta_j \right)  \prod_{{s=1},{s \neq j}}^N (1-\beta_s/\beta_j)}{h_{k_1,k_2,\ldots,k_N}(\beta)h_{l_1,l_2,\ldots,l_N}(1/\beta) }\\& \qquad \times \left(\frac{1}{2 \pi i}\oint_{\partial \mathbb{D}} h_{k_1,k_2,\ldots, k_N}(\beta;\beta_j\mapsto z) \prod_{\ell=1}^{s_1} F_{\ell}(z) \frac{dz}{z^{x_1+1}} \right) \\
& \qquad \quad \times \left( \frac{1}{2 \pi i}\oint_{\partial \mathbb{D}}  h_{l_1,l_2,\ldots,l_N}(1/\beta;1/\beta_j \mapsto 1/z) 
 \prod_{\ell=s_2+1}^{2N} F_{\ell}(z) \frac{z^{x_2}dz}{z} \right).
\end{split}
\eeq
If we assume in additon that  $s_1, s_2 \leq N$, then
the kernel $K(s_1,x_1;s_2,x_2)$ takes the simpler form
\beq \label{kernel1}
\begin{split}
&K(s_1, x_1; s_2, x_2) =-{\mathbb 1}_{s_1>s_2} \frac{1}{2 \pi i} \oint_{\partial \mathbb{D}} \prod_{j=s_2+1}^{s_1} (1-\alpha_j z)^{-1} \frac{dz}{z^{x_1-x_2+1}} \\
& \quad +\sum_{j=1}^N \frac{1}{2 \pi i}\oint_{\partial \mathbb{D}} \frac{h_{k_1,k_2,\ldots, k_{N}}(\beta; \beta_j \mapsto z)}{h_{k_1,k_2,\ldots, k_{N}}(\beta)} \prod_{r=1}^{s_2}(1-\alpha_r \beta_j)\beta_j^{x_2} \prod_{\ell=1}^{s_1} (1-\alpha_\ell z)^{-1} \frac{dz}{z^{x_1+1}}.
\end{split}
\eeq
\end{theorem}

\begin{remark*}
For the special case $l_j=j-1$ a different representation for the kernel in \eqref{kernel0} can be found in  \cite{ImSa}. Their proof is based on a computation of the inverse of the Gramm matrix in the Eynard-Mehta Theorem.  We follow a different approach that is better suited for our purposes. We first perform a preliminary bi-orthogonalization so that the corresponding Gramm matrix becomes diagonal.
\end{remark*}

\begin{remark*}
The condition $l_N-k_1\leq N-1$ is necessary in our approach in order to be able to get a diagonal Gramm matrix (see the discussion before \eqref{diagonalGramm}). Though it seems to be a technical condition, one way of looking at it is as follows: the two conditions $l_N-k_1 \leq N-1$ and $k_N-l_1 \leq N-1$ together are equivalent to having both starting and ending positions in a densely packed configuration. Thus, removing one of these conditions and keeping the other is a natural extension of the densely packed situation. One in which neither the starting nor the ending points need to be densely packed.
\end{remark*}

\begin{remark*}
It is remarkable that the kernel \eqref{kernel1} for $s_j\leq N$ does not depend on the endpoints $l_j$. Furthermore, the kernel only depends on the parameters $\alpha_j$ with $j\leq \max\{s_1,s_2\}$.  These effects seem to be related to the memorylessness of geometric random variables (recall that the height of the jumps at each step is geometrically weighted).
\end{remark*}

As noted in the Introduction, we are interested in the special case, $k_j=k(j-1)$, $l_j=j-1$, and we want to study the process near the left boundary, namely, for $s_1,s_2\leq N$. In fact, as remarked above, formula \eqref{kernel1} shows us that in this case, as long as $l_N \leq N-1$, the kernel is independent of the points on the right hand side. The following theorem shows that the kernel has an even simpler formula in this case.

\begin{theorem} \label{equalSpacing}
In the case $k_j=k(j-1)$  $($for some $k \in \bbN$, fixed$)$, and any $l_1<l_2< \ldots <l_N \leq N-1$, and for $(s_1, x_1), (s_2,x_2) \in \bbN^2$ with $s_1, s_2 \leq N$,
\beq \label{kernel2}
\begin{split}
& K(s_1,x_1;s_2,x_2)=-{\mathbb 1}_{s_1>s_2} \frac{1}{2 \pi i} \oint_{\partial \mathbb{D}} \prod_{j=s_2+1}^{s_1} (1-\alpha_j z)^{-1}\frac{dz}{z^{x_1-x_2+1}} \\
& \quad +\frac{k}{(2 \pi i)^2} \oint_{\partial \mathbb{D}} \oint_{\Gamma_\beta} \frac{\prod_{r=1}^{s_2}(1-\alpha_r w)\prod_{j=1}^N (z^k-\beta_j^k)}{\prod_{\ell=1}^{s_1}(1-\alpha_\ell z) \prod_{t=1}^N(w^k-\beta_t^k)}  
\frac{w^{x_2+k-1}dw dz}{z^{x_1+1} (z^k-w^k)},
\end{split}
\eeq
where the integration in $w$ is carried out on $\Gamma_\beta$, a contour around the $\beta_j$'s which avoids the other zeros of $w^k-\beta_j^k$. The integration in $z$ is on $\partial \mathbb{D}$.

In particular, in the case $\beta_r \equiv \beta$ for some $\beta$ with $| \beta| <1$, it follows that
\beq \label{kernel3}
\begin{split}
& K(s_1,x_1;s_2,x_2)=-{\mathbb 1}_{s_1>s_2} \frac{1}{2 \pi i} \oint_{\partial \mathbb{D}} \prod_{j=s_2+1}^{s_1} (1-\alpha_j z)^{-1}\frac{dz}{z^{x_1-x_2+1}} \\
& \quad +\frac{k}{(2 \pi i)^2} \oint_{\partial \mathbb{D}} \oint_{\Gamma_\beta} \frac{\prod_{r=1}^{s_2}(1-\alpha_r w) (z^k-\beta^k)^N}{\prod_{\ell=1}^{s_1}(1-\alpha_\ell z) (w^k-\beta^k)^N}  
\frac{w^{x_2+k-1}dw dz}{z^{x_1+1} (z^k-w^k)},
\end{split}
\eeq
\end{theorem}

\subsection{The limiting process}

Our next task is to compute the limiting behavior of  the process close to the starting points as $N\to \infty$.  We will only consider the equally spaced situation $k_j=k(j-1)$ and take $\alpha_j \in (0,1)$ and  $\beta_j=\beta$ for some constant $\beta\in (0,1)$. Hence we consider the process with the kernel in  \eqref{kernel3}. The region of interest to us is the following: take $s\in \bbN$ and  $\xi\in (0,1)$ and consider the process in the neighborhood of the point $(s,\xi k N)$ (note that the starting points are in the line segment $[0,kN]$). 

As noted in the beginning of the previous subsection, a determinantal point process is completely determined by the kernel. Hence, in order to find the limiting process it suffices to derive the limiting behavior of the kernel $K$ in \eqref{kernel3a}. Before we state our main result, we first define a family of relevant kernels.

\begin{definition}
 For $k \in \bbN$ and a sequence of positive numbers, $\gamma=\{\gamma_j\}_{j=1}^\infty$, let   $\bbK_k^\gamma$ be the kernel defined by 
\beq \label{kernel3a}
\begin{split}
\bbK_k^{\gamma} (s_1,x_1;s_2,x_2)
&=-\frac{{\mathbb 1}_{s_1>s_2}}{2\pi i} \oint_{\Gamma_0} \prod_{r=s_2+1}^{s_1} (1-\gamma_r z)^{-1} \frac{dz}{z^{x_1-x_2+1}}\\&+\sum_{j=0}^{k-1} \frac{\omega_k^{-j x_1}}{2\pi i} \int_{e^{-\pi i/k}}^{e^{\pi i/k}} \frac{\prod_{r=1}^{s_2}{(1- \gamma_r z)}}{\prod_{t=1}^{s_1}{(1-\omega_k^j \gamma_t z)}} \frac{dz}{z^{x_1-x_2+1}}.
\end{split}
\eeq
Here $\omega_k=e^{2\pi i/k}$, $\Gamma_0$ is a closed positively oriented contour around the pole $z=0$ and no other, and the integrals from $e^{-i \pi /k}$ to $e^{\pi i/k}$ are over a path that intersects the real axis only once in $(0,\inf_{j=s_2+1,\ldots,s_1}(1/\gamma_j))$. 
\end{definition}
To the best of our best knowledge the kernel $\bbK_k^\gamma$ has not appeared in the literature before. We will derive some properties later on, but first we state the main asymptotic result of the paper.

 \begin{theorem}\label{th3}

Let $\{x(N)\}_N \subset \bbN$ be such that 
\beq 
\label{eq:limitsequence}
x(N)=0 \mod k \quad \textrm{ and } \lim_{N \to \infty} \frac{x(N)}{kN}=\xi\in (0,1).
\eeq
Then, with $K$ as in \eqref{kernel3} and $\bbK_k^\gamma$ as in \eqref{kernel3a}, we have that
\beq \label{eq:limitth3}
\begin{split}
\lim_{N\to \infty} {((\xi/(1-\xi))^{1/k} \beta )}^{x_1-x_2}  K(s_1,x(N)+x_1; &  s_2,x(N)+x_2)  \\
&=\bbK_k^{\gamma } (s_1,x_1;s_2,x_2)
\end{split}
\eeq
for $(s_j,x_j)\in \bbN\times \bbZ$. Here $\gamma=(\gamma_1,\gamma_2,\ldots)$ with $\gamma_j=(\xi/(1-\xi))^{1/k} \beta \alpha_j$.
\end{theorem}

\begin{remark}
In the limit on the right-hand side of \eqref{eq:limitth3} the kernel is conjugated with $\left( (\xi/(1-\xi))^{1/k} \beta \right)^x$. This however does not change the process. Indeed, from \eqref{eq:detprocess} one readily checks that  if $K$ is the kernel of a determinantal point process and $G$ is a non-vanishing function, then $K_G(x,y)=G(x)/ G(y) \, K(x,y)$ is a different kernel for the same process.
\end{remark}

\begin{remark}
Putting $s_1=s_2=s$ 
and  $x_1=x_2=x$ we get the following expression for the mean density for the process induced by $\bbK_k^{\gamma}$.
\beq \label{kernel3c}
\begin{split}
\bbK_k^{ \gamma} (s,x;s,x)
=\frac{1}{k}+\sum_{j=1}^{k-1} \frac{\omega_k^{-j x}}{2\pi i} \int_{e^{-\pi i/k}}^{e^{\pi i/k}} \prod_{r=1}^{s}\left(\frac{1-\gamma_r z}{1-\omega_k^j \gamma_r z}\right) \frac{dz}{z}.
\end{split}
\eeq
This expression is clearly $k$-periodic in $x$. Moreover, 
Theorem \ref{th3} tells us that the leading order term for the mean density $K(s,x(N)+x;s,x(N)+x)$, as $N\to \infty$,  is given by $\bbK_k^{\gamma}(s,x;s,x)$ in \eqref{kernel3c}. In Figures \ref{fig:ex1} and \ref{fig:ex2} we used \eqref{kernel3c} to plot a local approximation to the mean density in some special cases. 
Namely, we write each integer on the horizontal axis as $j k+x$, with $0\leq x<k$, and compute the values of $\bbK_k^\gamma$ for $\xi=j/N$ (note that $\gamma$ depends on $\xi$ as indicated in Theorem \ref{th3}). The local $k$-periodicity is clearly visible in these graphs. In addition to this, the amplitude as a function of $\xi$ shows an interesting oscillatory behavior.  
\begin{figure}[t]
\begin{center}
\subfigure[$s=1$]{\includegraphics[width=.4 \textwidth,height=.2 \textheight]{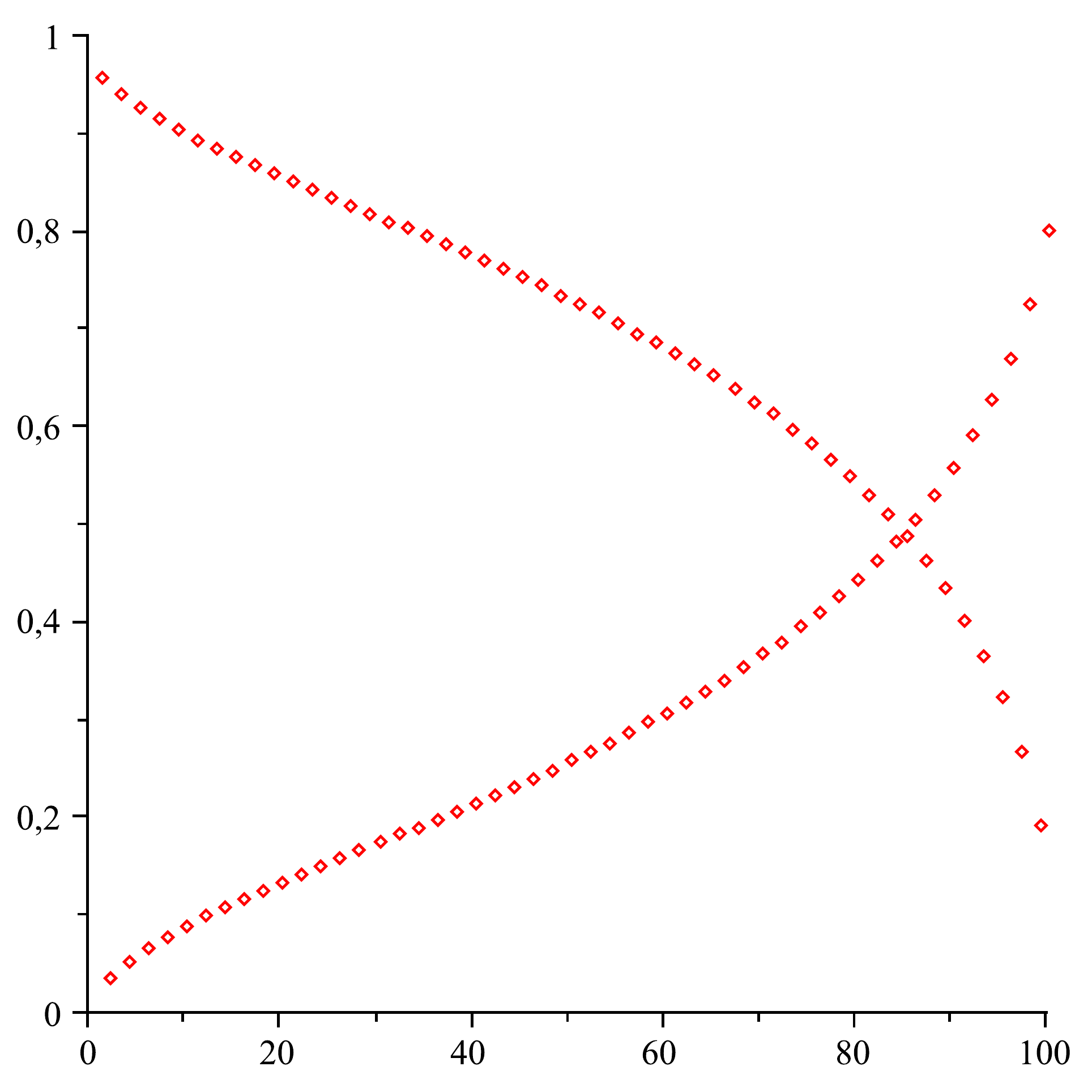}}
\subfigure[$s=3$]{\includegraphics[width=.4 \textwidth,height=.2 \textheight]{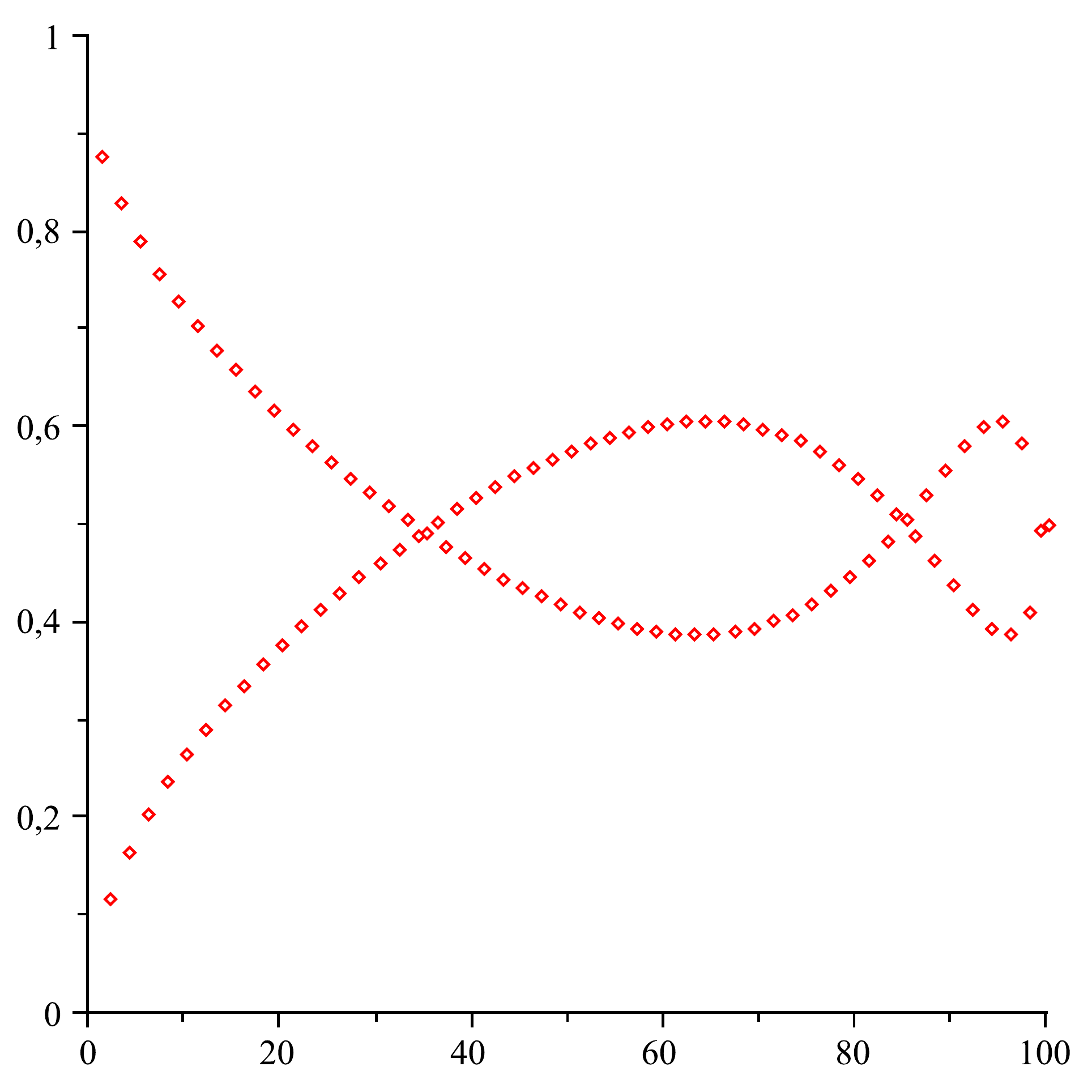}}\\
\subfigure[$s=5$]{\includegraphics[width=.4 \textwidth,height=.2 \textheight]{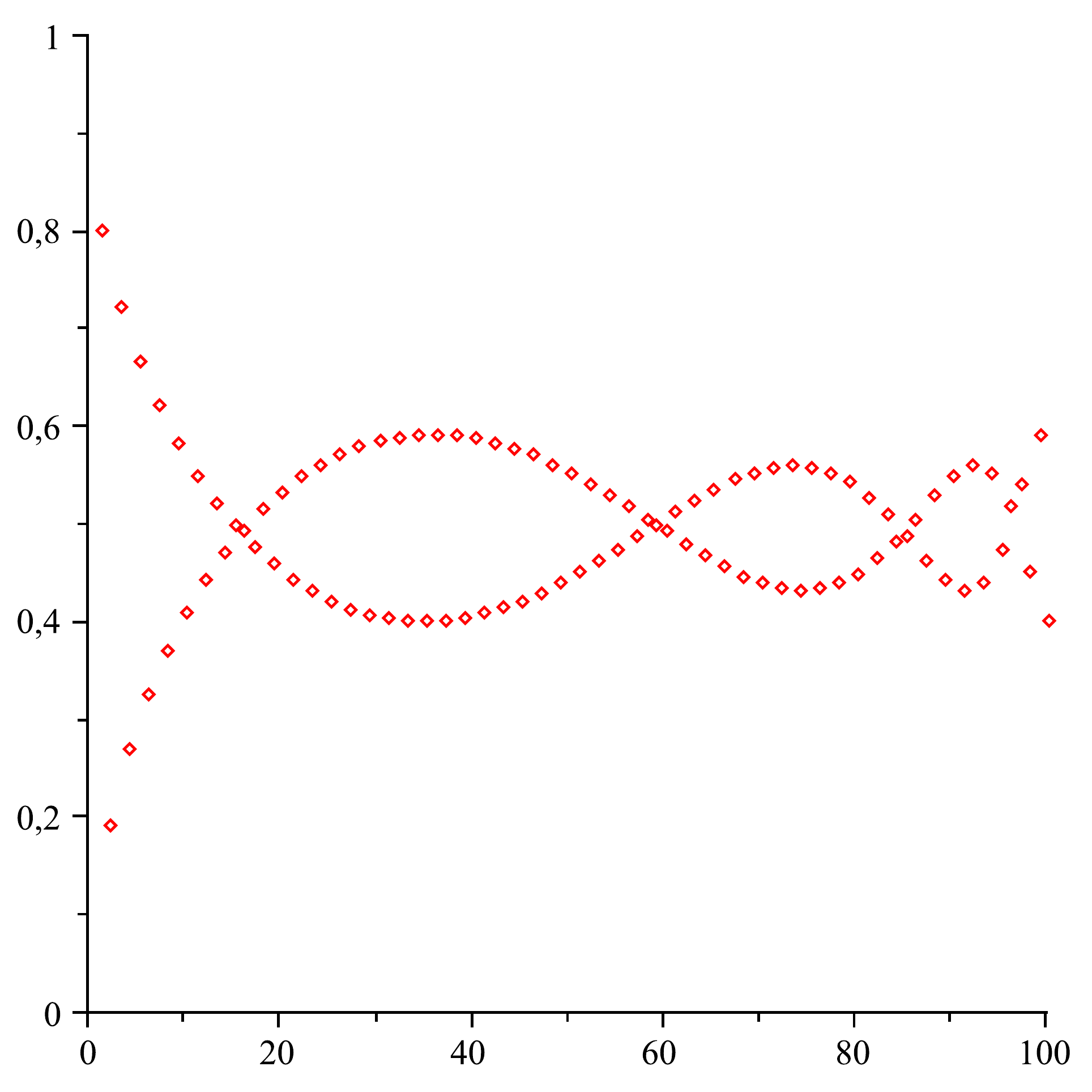}}
\subfigure[$s=7$]{\includegraphics[width=.4 \textwidth,height=.2 \textheight]{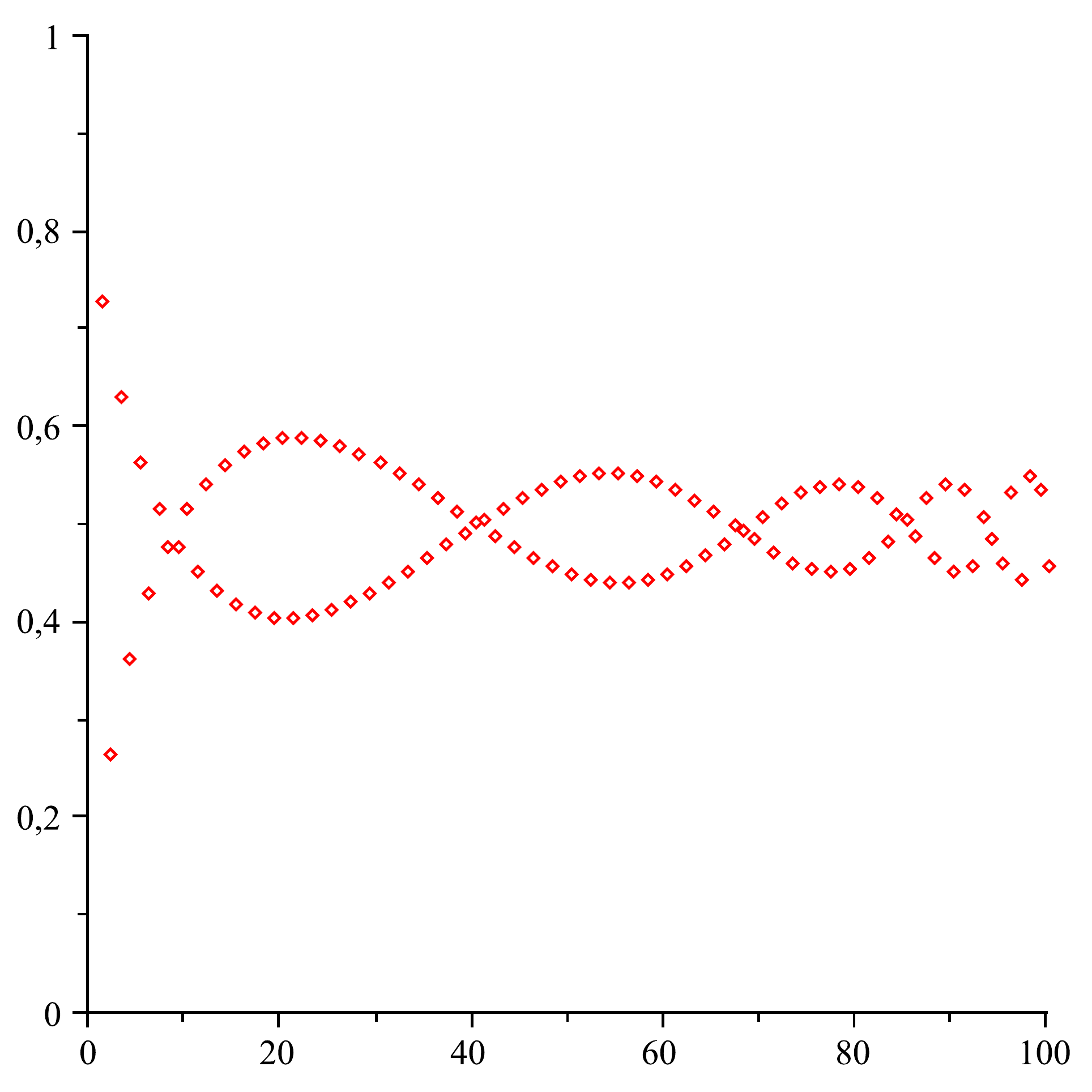}}
\caption{The leading order term in the limiting mean density in case $k=2$, $\alpha_j=\beta=2/3$, $N=50$ and four different values of $s$.  For each point $\xi N+x$, with $0 \leq x <k$, on the horizontal axis  the density is computed by using \eqref{kernel3c} locally for each value of $\xi=0,1/50,2/50,\ldots,1$.}
\label{fig:ex1}
\end{center}
\end{figure} 
\begin{figure}[t]
\begin{center}
\subfigure[$s=1$]{\includegraphics[scale=0.26]{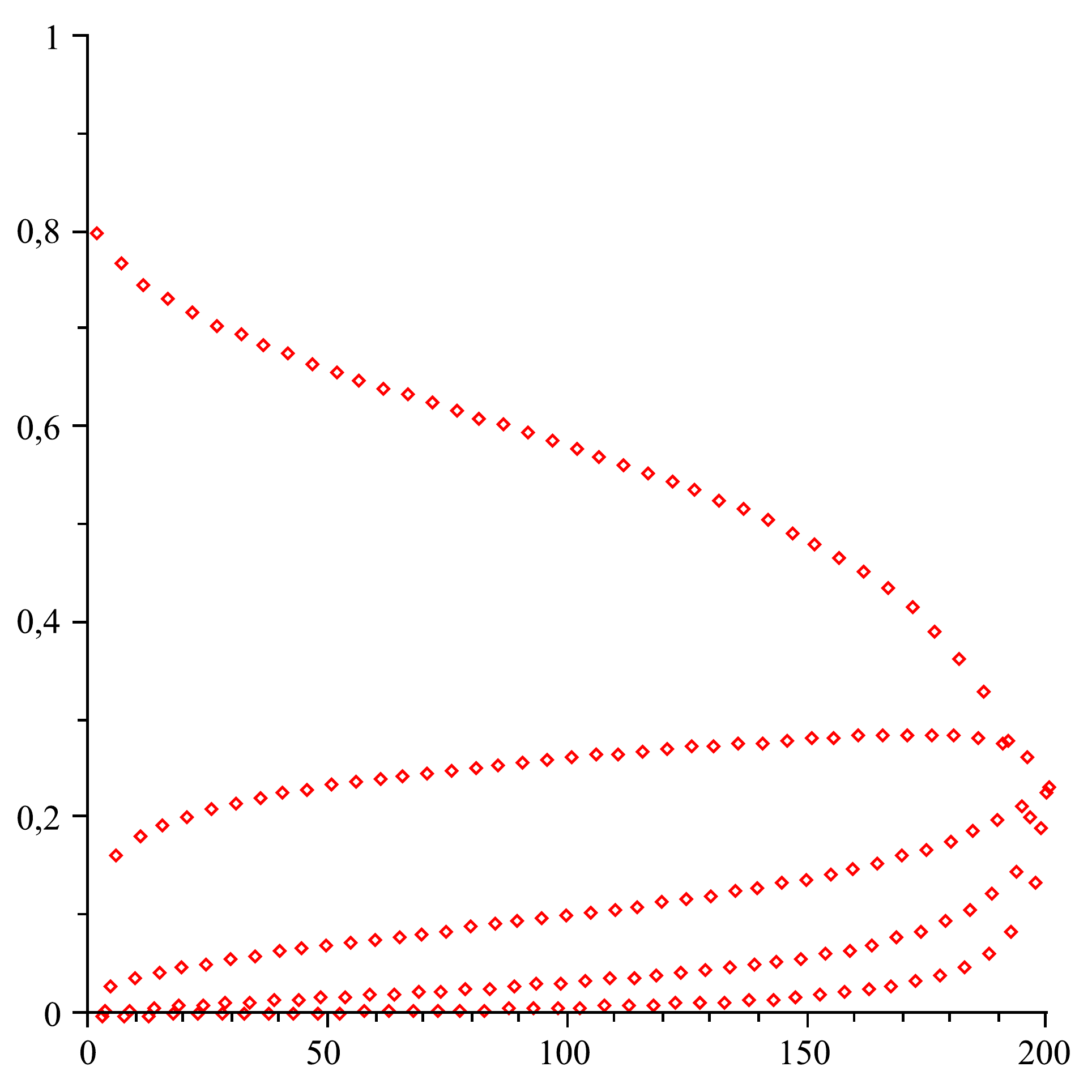}}
\subfigure[$s=5$]{\includegraphics[scale=0.26]{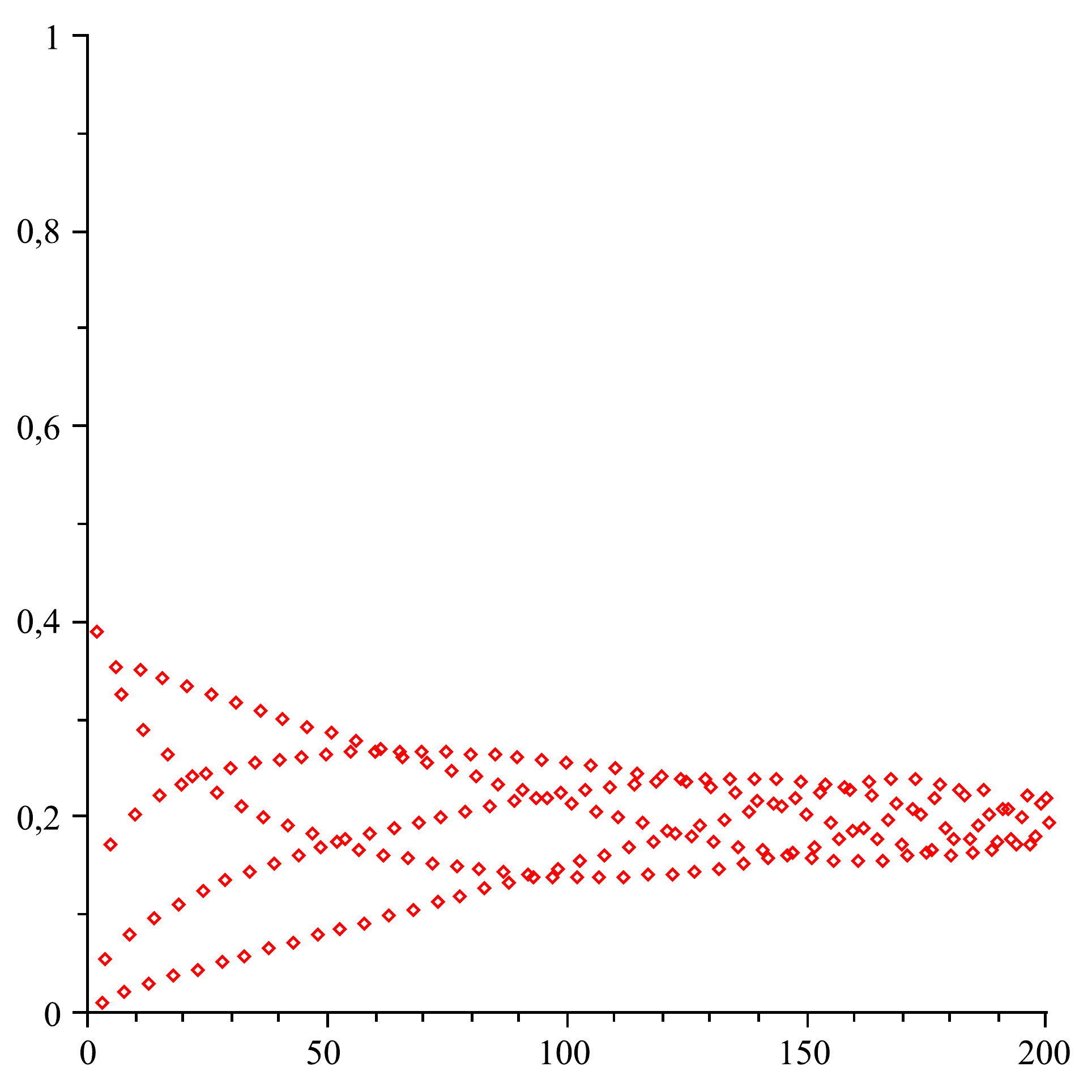}}
\caption{The leading order term in the limiting mean density in case $k=5$, $\alpha_j=\beta=2/3$, $N=40$ and two different values of $s$. Again, for each point $\xi N+x$, with $0 \leq x <k$,  on the horizontal axis  the density is computed by using \eqref{kernel3c} locally for each value of $\xi=0,1/40,2/40,\ldots,1$.}
\label{fig:ex2}
\end{center}
\end{figure} 
\end{remark}
\begin{remark}
There is a continuous time analogue for the kernel in \eqref{kernel3a}. Indeed, if we set $\gamma_j=1/S$, $s_j=\sigma_j S$ and take the limit $S \to \infty$ the kernel  converges to the kernel
\beq \label{kernel3cont}
\begin{split}
-\frac{{\mathbb 1}_{\sigma_1>\sigma_2}}{2\pi i} \oint_{\Gamma_0}  \frac{e^{(\sigma_1-\sigma_2)z}\ dz}{z^{x_1-x_2+1}}+\sum_{j=0}^{k-1} \frac{\omega_k^{-j x_1}}{2\pi i} \int_{e^{-\pi i/k}}^{e^{\pi i/k}} \frac{e^{ (\omega_k^j \sigma_1 - \sigma_2)z}\ dz}{z^{x_1-x_2+1}}.
\end{split}
\eeq

This kernel can be obtained by a limit similar to the one in Theorem \ref{th3} if we modify the model in the following way. Take the same conditions on the starting points and endpoints. Instead of geometric transition probabilities, we equip every walker with an exponential clock  and when the clock rings it jumps up by one and the clock resets. In the left half the jump by one is upwards and in the right half it is downwards.  Starting from the nonintersecting path model, the equivalent of Theorem \ref{th3} is with the kernel \eqref{kernel3cont} instead of $\bbK^\gamma_k$ in \eqref{kernel3a}. 
\end{remark}

\subsection{Limits of the new kernel $\mathbb K_k^\gamma$}

We will now present three limits of our family of kernels \eqref{kernel3a}. 

In the first limit we take the process away from the initial starting points and enter the  bulk region.  As it turns out, the kernel converges to  the (discrete) sine kernel. In fact, by tuning the parameters $\gamma_j$ we obtain a certain family of extensions of the sine kernel  that we will  describe now. Recall that the discrete sine kernel is given by 
\beq
K_{sine}^c(x,y)=\frac{\sin c(x-y)}{\pi(x-y)}=\frac{1}{2\pi i} \int_{e^{-ic}}^{e^{ic}} \frac{1}{z^{x-y+1}} dz, \qquad x,y\in \bbZ,
\eeq
where $c\in (0,\pi]$. Given a sequence $\{\gamma_j\}_j\subset (0,\infty)$  we can define an extension of the sine process by 
\beq\label{eq:extsinekernel}
K_{sine,ext}^{c,\gamma}(s,x;t,y)=
\begin{cases}
\frac{1}{2\pi i} \int_{\Gamma^+(e^{-ic},e^{ic})}  {\prod_{j=s+1}^t (1-\gamma_jz)} \frac{1}{z^{x-y+1} }dz, & s\leq t \\
\frac{1}{2\pi i} \int_{\Gamma^-({e^{-ic}},e^{ic})}  {\prod_{j=t+1}^s (1-\gamma_jz)^{-1}}  \frac{1}{z^{x-y+1}} dz, & s> t
\end{cases}
\eeq
where $\Gamma^+(e^{-ic},e^{ic})$ is a path from  $e^{-ic}$ to $e^{ic}$ that intersects the real axis once, at a point to the right of zero, and  $\Gamma^-(e^{-ic},e^{ic})$ goes from $e^{-ic}$ to $e^{ic}$ and intersects the real axis to the left of zero. The extensions \eqref{eq:extsinekernel} are a subclass of a more general family of extensions that are introduced in \cite{Bor}. Also note that if $\gamma_j\equiv \gamma$, then this is the incomplete  beta kernel as discussed in \cite{OR}.

\begin{proposition}\label{prop1}
Fix a sequence of positive numbers, $\gamma=(\gamma_1, \gamma_2, \ldots)$, and a number $0<\gamma<1$. For any $S \in \bbN$, let $\gamma_S=(\underbrace{\gamma,\ldots,\gamma}_{S \textrm{ times}}, \gamma_1,\gamma_2,\ldots)$.
Then
\beq
\lim_{S\to \infty} \bbK_k^{\gamma_S}  (S+s_1,x_2;S+s_2,x_2)=K_{sine,ext}^{\pi/k,\gamma} (s_1,x_1;s_2,x_2),
\eeq
for $(s_j,x_j)\in \bbN\times \bbZ$, where the right-hand side is the extension of the sine kernel as given in \eqref{eq:extsinekernel}.
\end{proposition}

\begin{remark}
We have a similar theorem for the continuous analogue \eqref{kernel3c}. By setting $\sigma_j=\sigma+\tilde \sigma_j$ and taking the limit $\sigma \to \infty$ we obtain a different extension of the sine kernel that also falls into the class described in \cite{Bor}.
\end{remark}

The second limit that we compute  is for $k\to \infty$. In this case the starting points for the paths are far apart. 

\begin{proposition}\label{prop2}
 Let $\bbK_k^\gamma$ be the kernel in \eqref{kernel3a}, with $\gamma_j \in (0,1)$ for all $j \in \bbN$.  Then 
\beq
\begin{split}
\lim_{k\to \infty}&\ \frac{ \prod_{r=1}^{s_1}{(1-\gamma_r )}}
{\prod_{r=1}^{s_2}{(1-\gamma_r )}}\ \bbK_k^{\gamma}(s_1,x_1;s_2,x_2)
\\&\quad =-\frac{{\mathbb 1}_{s_1>s_2}  \prod_{j=s_2+1}^{s_1} (1-\gamma_j)}{2\pi i} \oint_{\partial \mathbb D} \prod_{j=s_2+1}^{s_1} (1-\gamma_j w)^{-1} \frac{dw}{w^{x_1-x_2+1}}\\
&\qquad +{\frac{\prod_{r=1}^{s_1}(1- \gamma_r)}{2\pi i}\int_{\partial \mathbb D} }{\prod_{t=1}^{s_1}{(1-\gamma_t  w)^{-1}}} \ \frac{dw}{w^{x_1+1}}.
\end{split}
\eeq

Namely, as $k\to \infty$, the process induced by the kernel $\bbK_k^{\gamma}$ converges to a random walk of a single particle where the transition probabilities are given by the geometric distribution, with parameters $\gamma_j$. 
\end{proposition}
\begin{remark}
Without the restriction $\gamma_j <1$, the limiting process makes no sense as a probability measure. Note, in addition, that in Theorem \ref{th3} we have  $\gamma_j=(\xi/(1-\xi))^{1/k} \beta \alpha_j$ which converges to $\alpha_j \beta<1$ as $k\to \infty$. Thus, $\gamma_j<1$ is also natural from the viewpoint of the original process. Moreover, it is not hard to adapt the proof of Proposition \ref{prop2} for $\gamma_j$ that vary with $k$ in this way. \end{remark}
\begin{remark}
The conclusion that the process converges to a single random walk follows by a simple manipulation (see \eqref{eq:detmanip}) on the determinant in the $k$-point correlations \eqref{eq:detprocess}. Note that if $s_j=s$  the kernel clearly only depends on $x_1$ and not on $x_2$. This implies that $K$, viewed as an infinite matrix, has rank $1$ and hence the limiting process consists of a single particle only \cite[Theorem 4]{Sosh}. \end{remark}

The limits in  Propositions \ref{prop1} and \ref{prop2} are rather straightforward. A more interesting limit is $k,S\to \infty$ simultaneously. That is, we want to scale $S$ with (a power of) $k$ and take the limit $k\to \infty$. We restrict the discussion to the single time situation $s_1=s_2$ and take $\gamma_j \equiv \gamma$.  

In order to determine the proper scaling for $S$, consider a single random walker with geometrically distributed  transition probabilities with parameter $\gamma$. The mean position of the walker after $S$ steps is $S \gamma/(1-\gamma)$. The variance is $S \gamma^2/(1-\gamma)^2$. It follows that two independent random walkers starting at a distance $k$ will start noticing each other after approximately $k^2$ steps. Thus, the correct scaling is $S=\sigma k^2$, where the scaling parameter $\sigma$ does not depend on $k$. 

In order to obtain a meaningful limit, we want to scale the space variable, $x_j$, as well. The up-right paths force a drift in the random walks. Hence, if we want to \lq follow\rq \ the process as $k\to \infty$, we need to subtract the mean position of the walker. Moreover, the variance implies we need to scale the space variables with $k$. Hence 
\beq \no
x_j= \textrm{integer part of  } \Big(\sigma k^2 \gamma/(1-\gamma) +k \eta_j\big), \qquad j=1,2,
\eeq
where $\eta_j\in \bbR$.  Finally, because of the scaling with $k$ in the space variable, we also need to multiply the kernel $\bbK_k^\gamma$ with $k$  to get a meaningful answer.

\begin{proposition} \label{prop3}
Consider $\bbK_k^\gamma$ in \eqref{kernel3a} with $\gamma_j \equiv \gamma\in (0,1)$ and  
\beq \label{eq:newvariableeta}
\begin{cases}
s_j=s=\sigma k^2\\
x_j= \textrm{integer part of  } \Big(\sigma k^2 \gamma/(1-\gamma) +k \eta_j\Big).
\end{cases}\eeq
We have
\beq \label{kernelJoh}
\lim_{k\to \infty} k \ \bbK_k^\gamma (s,x_1;s,x_2)
=\frac{1}{\pi} \sum_{j=-\infty}^\infty   e^{(- \pi d j(j-1))}
\Re \left(\frac{ \exp\left(\pi i((2j-1)\eta_1+\eta_2)\right)}{i (\eta_2-\eta_1)+dj}\right),
\eeq
where $d=2\pi \sigma\gamma/(1-\gamma)^2  $.
\end{proposition}
The limiting kernel  \eqref{kernelJoh} appeared for the first time in \cite[eq (2.20)]{J3}. 

Note that in case  $\sigma$ is large (for $\sigma\to \infty$ we get the continuous sine kernel) then this kernel can be approximated by a simpler kernel  as indicated in \cite{J3}. Indeed, the terms other than $j=0,1$ are exponentially small as $\sigma$ (and hence $d$) is large.  Hence the right-hand side of \eqref{kernelJoh} can be approximated by the two  terms coming from $j=0,1$ which leads to
\beq
\frac{\sin \pi (\eta_1-\eta_2)}{\pi (\eta_1-\eta_2)} +\frac{d \cos \pi(\eta_1+\eta_2)+(\eta_2-\eta_1) \sin \pi (\eta_1+\eta_2)}{\pi (d^2+(\eta_1-\eta_2)^2)}.
\eeq
\subsection{Number variance}

An interesting feature of the kernel \eqref{kernel3a} is the phenomenon of number variance saturation. Let us fix $s_1=s_2=s$ and consider the number variance
\[\Var \left(\textrm{number of points in } [-L,L]\right), \]
that is, the variance of the random variable that counts the number of points in the interval $[-L,L]$ with respect to the process induced by the kernel $\bbK_k^\gamma$ in \eqref{kernel3}.  The number variance can be seen as a measure for the disorder of the system.

An important feature of the sine process is that the number variance for the interval $[-L,L]$ grows logarithmically with $L$ (more precisely, it grows as $\frac{1}{\pi^2}{\log L}$ as $L\to \infty$). In fact, it can be shown that the sine kernel is the kernel with slowest growth rate for the number variance among all translation invariant kernels  \cite{Sosh}.  

In \cite{J3} Johansson showed that, in contrast with the sine process, the number variance for the kernel \eqref{kernelJoh} `saturates'. Namely, it remains bounded as $L\to \infty$.

Due to the nonintersecting condition and the discreteness of the model, number variance saturation occurs for  $\bbK_k^\gamma$ as well. Rather than computing a formula for the limiting behavior for the number variance, we give an easy upper bound for the finite $N$ case. As this upper bound does not depend on $N$, it remains valid when taking the limit $N\to \infty$.

We first note that, by the nonintersecting property, each particle except for the top one can jump a vertical jump of size at most $k$ on the first step. Thus, an interval of size $L$ in the bulk on the line $(-N+1)$, will contain at most $(L+k)/k=L/k+1$ particles and at least $(L-k)/k=L/k-1$ particles. It follows that the number variance for $s=1$ is bounded by $4$. Similarly, up to the line $(-N+s)$, except for the $s$ topmost particles, all particles may increase their height by at most $sk$. Thus, the number variance for the line $-N+s$ is bounded by $4s^2$. Since the bounds are independent of $N$ (for intervals in the bulk), they persist as $N \rightarrow \infty$ and we get number variance saturation for the limiting process, for any fixed $s$. 


\section{Proofs}\label{sec:proofs}
In this section we prove Theorems \ref{generalFormula}, \ref{equalSpacing}, \ref{th3} and Propositions \ref{prop1}, \ref{prop2} and \ref{prop3}. 
\subsection{Proof of Theorem \ref{generalFormula}}
\begin{proof}[Proof of Theorem \ref{generalFormula}]

Let 
\beq \label{phi}
\Phi_j(x)=\frac{1}{2 \pi i} \oint_{\partial \mathbb{D}} z^{k_j} \frac{dz}{z^{x+1}} \qquad j=1,2,\ldots, N,
\eeq
\beq \label{psi}
\Psi_j(x)=\frac{1}{2 \pi i} \oint_{\partial \mathbb{D}} z^{x} \frac{dz}{z^{l_j+1}}  \qquad j=1,2,\ldots, N,
\eeq
and
\beq \label{T}
T_{\ell-1,\ell}(x_1,x_2)=\frac{1}{2 \pi i} \oint_{\partial \mathbb{D}} F_\ell(z) \frac{dz}{z^{x_2-x_1+1}} \qquad \ell=1,2,\ldots ,2N,
\eeq
where
\beq \label{Fell}
F_\ell(z)=\left \{ \begin{array}{cc}  &(1-\alpha_\ell z)^{-1} \quad 1 \leq \ell \leq N \\
&\left(1-\frac{\beta_{2N-\ell+1}}{z} \right)^{-1} \quad N+1 \leq \ell \leq 2N  \end{array} \right. .
\eeq

Note that $\Phi_j(y)$ are delta functions at $x=k_j$ and $\Psi_j(x)$ are delta functions at $l_j$ so they encode the boundary conditions of the walk. Moreover, $T_{\ell-1,\ell}(x_1,x_2)$ is the weight of the jump from $x_1$ on the line $s=\ell-1$ to $x_2$ on the line $s=\ell$. Thus, by the Lindstr\"om-Gessel-Vienot Theorem \cite{J,Stem}, the weight of the family of paths intersecting the line $s=\ell$ at the points $\{x_1^{(\ell)},x_2^{(\ell)},\ldots,x_N^{(\ell)}\}$ is given by 
\beq \label{LGV}
\begin{split}
& \det \left[ \Phi_i \left( x_j^{(0)} \right)\right]_{i,j=1}^N \det \left[T_{0,1} \left(x_i^{(0)},x_j^{(1)} \right) \right]_{i,j=1}^N \ldots \\
& \quad \times \det \left[T_{2N-1,2N}\left(x_i^{(2N-1)},x_j^{(2N)} \right) \right]_{i,j=1}^N \det \left[\Psi_i \left( x_j^{(2N)} \right) \right]_{i,j=1}^N.
\end{split}
\eeq

Thus, by the Eynard-Mehta Theorem \cite{BorDet}, the point process defined by the intersections of the paths with the lines $s=0,1,2,\ldots 2N$ is a determinantal point process whose kernel has the form 
\beq \label{EynardMehtaKernel}
\begin{split}
&K(s_1, x_1; s_2, x_2) =-{\mathbb 1}_{s_1>s_2} \frac{1}{2 \pi i} \oint_{\partial \mathbb{D}} \prod_{j=s_2+1}^{s_1} F_j(z) \frac{dz}{z^{x_1-x_2+1}} \\
& \quad +\sum_{i,j=1}^N \left[G_{i,j}^{-t} \right] \left(\frac{1}{2 \pi i}\oint_{\partial \mathbb{D}} z^{k_i} \prod_{\ell=1}^{s_1} F_{\ell}(z) \frac{dz}{z^{x_1+1}} \right) \\
& \qquad \times \left(\frac{1}{2 \pi i}\oint_{\partial \mathbb{D}} \prod_{\ell=s_2+1}^{2N} F_{\ell}(z) \frac{z^{x_2}dz}{z^{l_j+1}} \right),
\end{split}
\eeq
where the Gramm matrix is given by 
\beq \label{gramm}
G_{i,j}=\frac{1}{2\pi i} \oint_{\partial \mathbb{D}} \prod_{\ell=1}^{2N} F_{\ell}(z) \frac{z^{k_i} dz}{z^{l_j+1}}.
\eeq

Since the correlation functions are determinants of $K$, they are unchanged by replacing the $\Phi_i$'s by functions with the same linear span and the same is true of the  $\Psi_j$'s.  Let 
\beq \label{tildePhi}
\widetilde{\Phi}_j(x)=\frac{1}{2 \pi i}\oint_{\partial \mathbb{D}} h_{k_1,k_2,\ldots,k_N}(\beta;\beta_j \mapsto z) \frac{dz}{z^{x+1}},
\eeq
\beq \label{tildePsi}
\widetilde{\Psi}_j(x)=\frac{1}{2 \pi i} \oint_{\partial \mathbb{D}}  h_{l_1,l_2,\ldots,l_N}(1/\beta;1/\beta_j \mapsto 1/z) \frac{z^x dz}{z}.
\eeq 
Since we are assuming the $\beta_j$'s are distinct, $ h_{k_1,k_2,\ldots,k_N}(\beta;\beta_j \mapsto z)$ is a nontrivial linear combination of $z^{k_\ell}$, namely a polynomial in $z$, which vanishes at $\beta_\ell$ ($\ell \neq j$). Similarly, $h_{l_1,l_2,\ldots,l_N}(1/\beta, 1/\beta_j \mapsto 1/z)$ is a nontrivial linear combination of $z^{-l_\ell}$ which vanishes at $\beta_\ell$ ($\ell \neq j$). Moreover, 
\beq \label{spanTildePhi1}
\textrm{span}\left \{ z^{k_j} \right \}_{j=1}^N =\textrm{span} \left \{ h_{k_1,k_2,\ldots,k_N}(\beta;\beta_j \mapsto z) \right \}_{j=1}^N.
\eeq
Thus,
\beq \label{spanTildePhi2}
\textrm{span}\left \{ \Phi_j(x) \right \}_{j=1}^N =\textrm{span} \left \{ \widetilde{\Phi}_j(x) \right \}_{j=1}^N.
\eeq
Similarly
\beq \label{spanTildePsi1}
\textrm{span}\left \{ \frac{1}{z^{l_j}} \right \}_{j=1}^N =\textrm{span} \left \{ h_{l_1,l_2,\ldots,l_N}(1/\beta;1/\beta_j \mapsto 1/z) \right \}_{j=1}^N.
\eeq
and so 
\beq \label{spanTildePsi2}
\textrm{span}\left \{ \Psi_j(x) \right \}_{j=1}^N =\textrm{span} \left \{ \widetilde{\Psi}_j(x) \right \}_{j=1}^N.
\eeq

Thus, the kernel with $\Phi$ and $\Psi$ replaced by $\widetilde{\Phi}$ and $\widetilde{\Psi}$ describes the same process. Let us compute the Gramm matrix for this new kernel:
\beq \label{grammDiagonal}
\begin{split}
\widetilde{G}_{i,j} & =\frac{1}{2 \pi i} \oint_{\partial \mathbb{D}}   h_{k_1,k_2,\ldots,k_N}(\beta;\beta_i \mapsto z) h_{l_1,l_2,\ldots,l_N}(1/\beta;1/\beta_j \mapsto 1/z)\\
& \times \prod_{r=1}^N \left(1-\alpha_r z \right)^{-1} \left(1-\frac{\beta_r}{z} \right)^{-1}\frac{dz}{z} 
\end{split}
\eeq
The integral is over the unit circle (with counterclockwise orientation) and the integrand is meromorphic. Hence the integral can be computed by residue calculus. The only possible poles that are inside the unit circle are at $z=\beta_j$ for $j=1,\ldots,N$ and at $z=0$. However,  because of the condition $l_N-k_1\leq N-1$ there is no pole at $z=0$. Furthermore, due to the zeros of the functions $h_{k_1,k_2,\ldots,k_N}(\beta, \beta_i\mapsto z)$ and $h_{l_1,l_2,\ldots,l_N}(1/\beta,1/\beta_j\mapsto 1/z)$,  all poles are cancelled in case  $i\neq j$  and the integral vanishes. Hence we obtain
\beq \label{diagonalGramm}
\widetilde{G}_{i,j} =\left \{ \begin{array}{cc} 0 & \quad i \neq j \\
 \frac{h_{k_1,k_2,\ldots,k_N}(\beta)h_{l_1,l_2,\ldots,l_N}(1/\beta) }{\prod_{r=1}^N \left(1-\alpha_r \beta_j \right) \prod_{{s=1},{s \neq j}}^N (1-\beta_s/\beta_j)} & \quad i=j \end{array} \right. .
\eeq
So we see that the Gramm matrix is diagonal in this case. Inserting in \eqref{EynardMehtaKernel}, we get
\beq \label{EMkernel}
\begin{split}
&K(s_1, x_1; s_2, x_2) =-{\mathbb 1}_{s_1>s_2} \frac{1}{2 \pi i} \oint_{\partial \mathbb{D}} \prod_{j=s_2+1}^{s_1} F_j(z) \frac{dz}{z^{x_1-x_2+1}} \\
& \quad +\sum_{j=1}^N  \frac{\prod_{r=1}^N \left(1-\alpha_r \beta_j \right)  \prod_{{s=1},{s \neq j}}^N (1-\beta_s/\beta_j)}{h_{k_1,k_2,\ldots,k_N}(\beta)h_{l_1,l_2,\ldots,l_N}(1/\beta) }\\& \qquad \times \left(\frac{1}{2 \pi i}\oint_{\partial \mathbb{D}} h_{k_1,k_2,\ldots, k_N}(\beta;\beta_j\mapsto z) \prod_{\ell=1}^{s_1} F_{\ell}(z) \frac{dz}{z^{x_1+1}} \right) \\
& \qquad \quad \times \left( \frac{1}{2 \pi i}\oint_{\partial \mathbb{D}}  h_{l_1,l_2,\ldots,l_N}(1/\beta;1/\beta_j \mapsto 1/z) 
 \prod_{\ell=s_2+1}^{2N} F_{\ell}(z) \frac{z^{x_2}dz}{z} \right).
\end{split}
\eeq
and this is \eqref{kernel0}. 

To prove \eqref{kernel1}, we note that since $s_2 \leq N$ and $l_N \leq N-1$ (recall we assume $k_1=0$), we have 
\beq \label{Integration}
\begin{split}
& \frac{1}{2 \pi i}\oint_{\partial \mathbb{D}} h_{l_1,l_2,\ldots,l_N}(1/\beta;1/\beta_j \mapsto 1/z) 
 \prod_{\ell=s_2+1}^{2N} F_{\ell}(z) \frac{z^{x_2}dz}{z} \\
&\quad  =\frac{1}{2 \pi i}\oint_{\partial \mathbb{D}}
 \frac{h_{l_1,l_2,\ldots,l_N}(1/\beta;1/\beta_j \mapsto 1/z) 
}{\prod_{\ell=s_2+1}^N \left(1-\alpha_\ell z \right) \prod_{r=1}^N(1-\beta_r/z)}{z^{x_2-1}dz}\\
&\qquad = \frac{\beta_j^{x_2} h_{l_1,l_2,\ldots,l_N}(1/\beta) }{\prod_{r=1}^N \left(1-\alpha_r \beta_j \right) \prod_{{s=1},{s \neq j}}^N (1-\beta_s/\beta_j)} .
\end{split}
\eeq
Noting that $\frac{\prod_{r=1}^N \left(1-\alpha_r \beta_j \right)}{\prod_{\ell=s_2+1}^N \left(1-\alpha_\ell \beta_j \right)}=\prod_{r=1}^{s_2}(1-\alpha_r \beta_j)$, \eqref{kernel1} follows from inserting \eqref{Integration} into \eqref{EMkernel}.
\end{proof}


\subsection{Proof of Theorem \ref{equalSpacing}}
\begin{proof}[Proof of Theorem \ref{equalSpacing}]

We first assume the $\beta_j$'s are distinct. We may then apply Theorem \ref{generalFormula}. By noting that 
$h_{0,k,2k,\ldots,(N-1)k}(\beta; \beta_j \mapsto z)=C \prod_{\ell=1, \ell \neq j}^N(z^k-\beta_\ell^k)$ for a constant $C$ that is independent of $z$, we see that 
\beq \no
\frac{h_{0,k,2k,\ldots,(N-1)k}(\beta; \beta_j \mapsto z)}{h_{0,k,2k,\ldots,(N-1)k}(\beta)}=\frac{\prod_{\ell=1, \ell \neq j}^N \left(z^k-\beta_\ell^k \right)}{\prod_{\ell=1, \ell \neq j}^N \left(\beta_j^k-\beta_\ell^k \right)},
\eeq
and so
\beq \label{EMkernelApplied}
\begin{split}
&K(s_1, x_1; s_2, x_2) =-{\mathbb 1}_{s_1>s_2} \frac{1}{2 \pi i} \oint_{\partial \mathbb{D}} \prod_{j=s_2+1}^{s_1} F_j(z) \frac{dz}{z^{x_1-x_2+1}} \\
& \quad +\sum_{j=1}^N \frac{1}{2 \pi i}\oint_{\partial \mathbb{D}} \frac{\prod_{\ell=1, \ell \neq j}^N \left(z^k-\beta_\ell^k \right)}{\prod_{\ell=1, \ell \neq j}^N \left(\beta_j^k-\beta_\ell^k \right)} \prod_{r=1}^{s_2}(1-\alpha_r \beta_j)\beta_j^{x_2} \prod_{\ell=1}^{s_1} F_{\ell}(z) \frac{dz}{z^{x_1+1}}\\
&=-{\mathbb 1}_{s_1>s_2} \frac{1}{2 \pi i} \oint_{\partial \mathbb{D}} \prod_{j=s_2+1}^{s_1} F_j(z) \frac{dz}{z^{x_1-x_2+1}} \\
& \quad +\sum_{j=1}^N \frac{1}{2 \pi i}\oint_{\partial \mathbb{D}} \prod_{\ell=1}^N \left(z^k-\beta_\ell^k \right)\prod_{\ell=1}^{s_1} F_{\ell}(z) \frac{\prod_{r=1}^{s_2}(1-\alpha_r \beta_j)\beta_j^{x_2}} {\prod_{\ell=1, \ell \neq j}^N \left(\beta_j^k-\beta_\ell^k \right)\left(z^k-\beta_j^k \right)} \frac{dz}{z^{x_1+1}}.
\end{split}
\eeq

Now, if $\Gamma_{\beta_j}$ is a small circle around $\beta_j$ which does not contain the other zeros of $w^k-\beta_j^k$, then, (recall $w^k-\beta_j^k=(w-\beta_j)(w^{k-1}+\beta_j w^{k-2}+\ldots \beta_j^{k-1}w +\beta_j^{k-1})$),
\beq \label{residue}
\begin{split}
& \frac{k}{2 \pi i} \oint _{\Gamma_{\beta_j}}  \frac{\prod_{r=1}^{s_2}(1-\alpha_r w)w^{x_2} w^{k-1}} {\prod_{\ell=1}^N \left(w^k-\beta_\ell^k \right)\left(z^k-w^k \right)}dw \\
& \quad = \frac{1}{2 \pi i} \oint _{\Gamma_{\beta_j}}  \frac{\prod_{r=1}^{s_2}(1-\alpha_r w)w^{x_2}} {\prod_{\ell=1, \ell \neq j}^N \left(w^k-\beta_\ell^k \right)\left(z^k-w^k \right)}\frac{k w^{k-1}dw}{w^k-\beta_j^k} \\
& \quad = \frac{\prod_{r=1}^{s_2}(1-\alpha_r \beta_j)\beta_j^{x_2}} {\prod_{\ell=1, \ell \neq j}^N \left(\beta_j^k-\beta_\ell^k \right)\left(z^k-\beta_j^k \right)}\frac{k \beta_j^{k-1}}{\beta_j^{k-1}+\beta_j \beta_j^{k-2} +\ldots+\beta_j^{k-1}} \\
& \quad = \frac{\prod_{r=1}^{s_2}(1-\alpha_r \beta_j)\beta_j^{x_2}} {\prod_{\ell=1, \ell \neq j}^N \left(\beta_j^k-\beta_\ell^k \right)\left(z^k-\beta_j^k \right)}.
\end{split}
\eeq

It follows that 
\beq \label{improvedKernel}
\begin{split}
& \sum_{j=1}^N \frac{1}{2 \pi i}\oint_{\partial \mathbb{D}} \prod_{\ell=1}^N \left(z^k-\beta_\ell^k \right)\prod_{\ell=1}^{s_1} F_{\ell}(z) \frac{\prod_{r=1}^{s_2}(1-\alpha_r \beta_j)\beta_j^{x_2}} {\prod_{\ell=1, \ell \neq j}^N \left(\beta_j^k-\beta_\ell^k \right)\left(z^k-\beta_j^k \right)} \frac{dz}{z^{x_1+1}} \\
& \quad = \frac{k}{(2 \pi i)^2} \oint_{\partial \mathbb{D}} \oint_{\Gamma_\beta} \frac{\prod_{j=1}^N (z^k-\beta_j^k)}{\prod_{\ell=1}^{s_1}(1-\alpha_\ell z)} \frac{\prod_{r=1}^{s_2}(1-\alpha_r w)}{\prod_{t=1}^N(w^k-\beta_t^k)} 
\frac{w^{x_2+k-1}dw dz}{z^{x_1+1} (z^k-w^k)},
\end{split}
\eeq
which, together with \eqref{EMkernelApplied}, is \eqref{kernel2}. 

In the case that the $\beta_j$'s are not distinct the same formula holds by continuity. Formula \eqref{kernel3} is immediate from \eqref{kernel2}.
\end{proof}


\subsection{Proof of Theorem \ref{th3}}

\begin{proof}[Proof of Theorem \ref{th3}]

The proof is based on a steepest descent argument (see \cite{O} for a discussion of this technique for kernels with double integral representations). Note that in the double integral representation \eqref{kernel2} we have an interaction term $1/(z^k-w^k)$ instead of the more common $1/(z-w)$. This gives rise to extra complications in the analysis.

We start with the kernel $K$ 
\beq \label{kernel2a}
\begin{split}
& K(s_1,x_1;s_2,x_2)=-{\mathbb 1}_{s_1>s_2} \frac{1}{2 \pi i} \oint_{\partial \mathbb{D}} \prod_{j=s_2+1}^{s_1} (1-\alpha_j z)^{-1}\frac{dz}{z^{x_1-x_2+1}} \\
& \quad +\frac{k}{(2 \pi i)^2} \oint_{\partial \mathbb{D}} \oint_{\Gamma_\beta} \frac{\prod_{r=1}^{s_2}(1-\alpha_r w) (z^k-\beta^k)^N}{\prod_{\ell=1}^{s_1}(1-\alpha_\ell z) (w^k-\beta^k)^N}  
\frac{w^{x_2+k-1}dw dz}{z^{x_1+1} (z^k-w^k)}.
\end{split}
\eeq
By substituting
\beq 
\frac{kw^{k-1}}{z^k-w^k} =\sum_{j=0}^{k-1} \frac{1}{\omega_k^{-j} z-w}, \quad \omega_k=e^{2 \pi i/k},
\eeq 
in \eqref{kernel2a},  the change of variable $\omega_k^{-j}z\mapsto z$ and the rotational invariance of $\partial \mathbb D$ we obtain
\beq \label{kernel2b}
\begin{split}
& K(s_1,x_1;s_2,x_2)=-{\mathbb 1}_{s_1>s_2} \frac{1}{2 \pi i} \oint_{\partial \mathbb{D}} \prod_{j=s_2+1}^{s_1} (1-\alpha_j z)^{-1}\frac{dz}{z^{x_1-x_2+1}} \\
& \quad +\sum_{j=0}^{k-1}\frac{\omega_k^{-jx_1}}{(2 \pi i)^2} \oint_{\partial \mathbb{D}} \oint_{\Gamma_\beta} \frac{\prod_{r=1}^{s_2}(1-\alpha_r w) (z^k-\beta^k)^N}{\prod_{\ell=1}^{s_1}(1 -\omega_k^j\alpha_\ell z) (w^k-\beta^k)^N}  
\frac{w^{x_2}dw dz}{z^{x_1+1} (z-w)}.
\end{split}
\eeq
Replacing $x_j$ with $x(N)+x_j$ gives 
\beq \label{kernel2c}
\begin{split}
& K(s_1,x(N)+x_1;s_2,x(N)+x_2)=-{\mathbb 1}_{s_1>s_2} \frac{1}{2 \pi i} \oint_{\partial \mathbb{D}} \prod_{j=s_2+1}^{s_1} (1-\alpha_j z)^{-1}\frac{dz}{z^{x_1-x_2+1}} \\
& \quad +\sum_{j=0}^{k-1}\frac{\omega_k^{-j x_1}}{(2 \pi i)^2} \oint_{\partial \mathbb{D}} \oint_{\Gamma_\beta} \frac{\prod_{r=1}^{s_2}(1-\alpha_r w)}{\prod_{\ell=1}^{s_1}(1 -\omega_k^j\alpha_\ell z)} 
\frac{w^{x_2}}{z^{x_1+1} (z-w)}  \frac{ (w^k-\beta^k)^{-N}w^{x(N)}}{ (z^k-\beta^k)^{-N} z^{x(N)}} dw dz.
\end{split}
\eeq
Note that  the terms in the integrand that vary with $N$ can be written as
\beq
\frac{z^{x(N)}}{(z^k-\beta^k)^N}= \exp\left(-N\log(z^k-\beta^k)+x(N) \log z\right),
\eeq
and similarly for $w$. Since $x(N)=\xi k N (1+\mathcal O(1/N)$ as $N\to \infty$, in order to derive the leading terms in the asymptotics of \eqref{kernel2c}, we need to find the saddle points of the function $F$ defined by
\beq \label{eq:defF}
F(z)=G(z^k), \quad G(z)=-\log(z-\beta^k) + \xi \log z.
\eeq
Clearly, the saddle points of $F$ are $k$-th roots of the saddle points of $G$. A straightforward calculation shows that $G$ has only one saddle point $\beta^k \xi/(\xi-1)$, which is simple. Hence there are four paths of steepest descent/ascent for $\Re G$ leaving from the saddle point. The paths of steepest descent together form the negative real axis. The paths of steepest ascent start in the saddle point and end in $z=\beta^k$ which is  a singular point of $G$. The saddle points and corresponding points of $F$ can now easily be found by taking the $k$-th roots.  In Figure \ref{fig:saddle} the saddle points and paths of steepest descent/ascent are illustrated in the case $k=4$.

We now deform the contours $\partial \mathbb{D}$ and $\Gamma_\beta$ in the following way.  The contour $\Gamma_\beta$ is deformed to two rays 
\beq
\Gamma_\beta=\bbR_+ e^{\pi i /k}\cup \bbR_+ e^{-\pi i /k},
\eeq
which consist of the paths of steepest descent for $\Re F$, with $F$ as in \eqref{eq:defF}, leaving from the saddle points $z^*=\beta (\xi/(1-\xi))^{1/k} e^{\pi i/k}$ and $\overline{z^*}$.  

We deform $\partial \mathbb D$ to a contour $\Gamma_0$ that consists of the paths of steepest ascent for $\Re F$ leaving from all the saddle points of $F$, (see also Figure \ref{fig:saddle}). Note that every such path starts at a  saddle point and ends at one of the critical point $\omega_k^j\beta$ of $F$. Together they form   a closed contour around the origin. It is important to note that the other singularities $1/\alpha_j$ are not encircled by $\Gamma_0$.

Note that that the contours $\Gamma_0$ and $\Gamma_\beta$ intersect. Hence by deforming the contours we pick up  residues due to the term $1/(z-w)$. As result we obtain an extra single integral in the representation for the kernel
\beq \label{kernel2d}
\begin{split}
& K(s_1,x_1;s_2,x_2)=-{\mathbb 1}_{s_1>s_2} \frac{1}{2 \pi i} \oint_{\partial \mathbb{D}} \prod_{j=s_2+1}^{s_1} (1-\alpha_j z)^{-1}\frac{dz}{z^{x_1-x_2+1}} \\
& \quad +\sum_{j=0}^{k-1}\frac{\omega_k^{-jx_1}}{2 \pi i} \int_{\overline{z^*} }^{z^*}  \frac{\prod_{r=1}^{s_2}(1-\alpha_r z)}{\prod_{\ell=1}^{s_1}(1-\omega_k^{j}\alpha_\ell z) }
\frac{dz}{z^{x_1-x_2+1} }\\
& \quad +\sum_{j=0}^{k-1}\frac{\omega_k^{-j x_1}}{(2 \pi i)^2} \oint_{\Gamma_0} \oint_{\Gamma_\beta} \frac{\prod_{r=1}^{s_2}(1-\alpha_r w)}{\prod_{\ell=1}^{s_1}(1 -\omega_k^j\alpha_\ell z)} 
\frac{w^{x_2}}{z^{x_1+1} (z-w)}  \frac{ (w^k-\beta^k)^{-N}w^{x(N)}}{ (z^k-\beta^k)^{-N} z^{x(N)}} dw dz.
\end{split}
\eeq
Here the integrals from $\overline{z^*}$ to $z^*$ are over the part of $\Gamma_0$ that is contained in $\Gamma_\beta$ or any homotopic deformation of that  path in $\bbC\setminus\{0,\alpha_j^{-1}\}$.

We claim that the double integral is of order $\mathcal O (N^{-1/2})$ as $N\to \infty$ (in fact, it can be shown to be $\mathcal O (N^{-1})$). This  follows by standard steepest descent arguments, so we restrict ourselves to a brief discussion that justifies the claim. The leading term in the asymptotic expansion for the double integral in \eqref{kernel2d} comes from small neighborhoods around the saddle points $z^*$ and $\overline{z^*}$. The other parts give only exponentially small contributions. In each of these neighborhoods one then introduces new local variables. For example, in the parts of the integrals that are both near $z^*$ we introduce the variables $w=z^*+ s N^{-1/2}$ and $z=z^*+t N^{-1/2}$. In these new variables the integrand converges to a Gaussian as $N\to \infty$. Because of the scaling coming from the change of variables, the double integral then behaves like a (double) Gaussian integral multiplied by $N^{-1/2}$. Hence we arrive at the claim.  

It follows that  
\beq\label{kernel2e}
\begin{split}
&\lim_{N\to \infty}  K(s_1,x_1;s_2,x_2)=-{\mathbb 1}_{s_1>s_2} \frac{1}{2 \pi i} \oint_{\partial \mathbb{D}} \prod_{j=s_2+1}^{s_1} (1-\alpha_j z)^{-1}\frac{dz}{z^{x_1-x_2+1}} \\
& \quad +\sum_{j=0}^{k-1}\frac{\omega_k^{-jx_1}}{2\pi i} \int_{\overline{z^*} }^{z^*}  \frac{\prod_{r=1}^{s_2}(1-\alpha_r z)}{\prod_{\ell=1}^{s_1}(1-\omega_k^{j}\alpha_\ell z) }
\frac{dz}{z^{x_1-x_2+1} }.
\end{split}
\eeq
By changing the integration variable to $w=z/|z^*|$  we obtain
\beq \label{kernel2f}
\begin{split}
&\lim_{N\to \infty}  |z^*|^{x_1-x_2}K(s_1,x_1;s_2,x_2)=-{\mathbb 1}_{s_1>s_2} \frac{1}{2 \pi i} \oint_{\partial \mathbb{D}} \prod_{j=s_2+1}^{s_1} (1-\gamma_j w)^{-1}\frac{dw}{w^{x_1-x_2+1}} \\
& \quad +\sum_{j=0}^{k-1}\frac{\omega_k^{-jx_1}}{2\pi i} \int_{\overline{e^{-\pi i /k}} }^{e^{\pi i/k}}  \frac{\prod_{r=1}^{s_2}(1-\gamma_r w)}{\prod_{\ell=1}^{s_1}(1-\omega_k^{j}\gamma_\ell w) }
\frac{dw}{w^{x_1-x_2+1} } 
\end{split}
\eeq
which is the statement.
 \end{proof}

\begin{figure}[t]
\begin{center}
\begin{tikzpicture}[scale=0.5]
\draw[dashed,very thick] (0,0)--(4,4);
\draw[-,very thick] (2,2) .. controls (2.5,1.5) and (3,0.5) .. (2.5,0);
\draw[-,very thick] (2,2) .. controls (1.5,2.5) and (0.5,3) .. (0,2.5);
\draw[dashed,very thick,rotate=90] (0,0)--(4,4);
\draw[-,rotate=90,very thick] (2,2) .. controls (2.5,1.5) and (3,0.5) .. (2.5,0);
\draw[-,rotate=90,very thick] (2,2) .. controls (1.5,2.5) and (0.5,3) .. (0,2.5);
\draw[dashed,rotate=180,very thick] (0,0)--(4,4);
\draw[-,rotate=180,very thick] (2,2) .. controls (2.5,1.5) and (3,0.5).. (2.5,0);
\draw[-,rotate=180,very thick] (2,2) .. controls (1.5,2.5) and (0.5,3) .. (0,2.5);
\draw[dashed,rotate=270,very thick] (0,0)--(4,4);
\draw[-,rotate=270,very thick] (2,2) .. controls (2.5,1.5) and (3,0.5).. (2.5,0);
\draw[-,rotate=270,very thick] (2,2) .. controls (1.5,2.5) and (0.5,3) .. (0,2.5);
\draw[-] (-5,0)--(5,0);
\draw[-] (0,-5)--(0,5);
\fill  (2,2)   circle (0.1cm); 
\fill  (2,-2)   circle (0.1cm); 
\fill  (-2,2)   circle (0.1cm); 
\fill  (-2,-2)   circle (0.1cm);
\fill (0,2.5)    circle (0.1cm); 
\fill (0,-2.5)    circle (0.1cm); 
\fill (2.5,0)    circle (0.1cm); 
\fill (-2.5,0)    circle (0.1cm); 
\end{tikzpicture}
\begin{tikzpicture}[scale=0.5]
\draw[dashed,very thick] (0,0)--(4,4);
\draw[<-,very thick] (1,1)--(1.1,1.1);
\draw[-,very thick] (2,2) .. controls (2.5,1.5) and (3,0.5) .. (2.5,0);
\draw[-,very thick] (2,2) .. controls (1.5,2.5) and (0.5,3) .. (0,2.5);
\draw[->,very thick] (0.6,2.7)--(0.5,2.7);
\draw[<-,very thick] (0.6,-2.7)--(0.5,-2.7);
\draw (3,4) node {$\Gamma_\beta$};
\draw[-,rotate=90,very thick] (2,2) .. controls (2.5,1.5) and (3,0.5) .. (2.5,0);
\draw[-,rotate=90,very thick] (2,2) .. controls (1.5,2.5) and (0.5,3) .. (0,2.5);
\draw (-2.75,-2.75) node {$\Gamma_0$};
\draw[-,rotate=180,very thick] (2,2) .. controls (2.5,1.5) and (3,0.5).. (2.5,0);
\draw[-,rotate=180,very thick] (2,2) .. controls (1.5,2.5) and (0.5,3) .. (0,2.5);
\draw[dashed,rotate=270,very thick] (0,0)--(4,4);
\draw[->,very thick] (1,-1)--(1.1,-1.1);
\draw[-,rotate=270,very thick] (2,2) .. controls (2.5,1.5) and (3,0.5).. (2.5,0);
\draw[-,rotate=270,very thick] (2,2) .. controls (1.5,2.5) and (0.5,3) .. (0,2.5);
\draw[-] (-5,0)--(5,0);
\draw[-] (0,-5)--(0,5);
\fill  (2,2)   circle (0.2cm); 
\draw (3,2) node {$z^*$};
\fill  (2,-2)   circle (0.2cm); 
\draw (3,-2) node {$\overline{z^*}$};
\fill  (-2,2)   circle (0.1cm); 
\fill  (-2,-2)   circle (0.1cm);
\fill (0,2.5)    circle (0.1cm); 
\fill (0,-2.5)    circle (0.1cm); 
\fill (2.5,0)    circle (0.1cm); 
\fill (-2.5,0)    circle (0.1cm); 
\end{tikzpicture}
\end{center}
\caption{The left picture shows the  paths of steepest descent and ascent for $\Re F$ leaving from the saddle points of $F$. The right picture shows the deformed contours $\Gamma_0$ and $\Gamma_\beta$ which consist of paths of steepest descent and ascent leaving from the saddle points $z^*=\beta (\xi/(1-\xi))^{1/k} e^{\pi i/k}$ and $\overline{z^*}$}
\label{fig:saddle}
\end{figure}
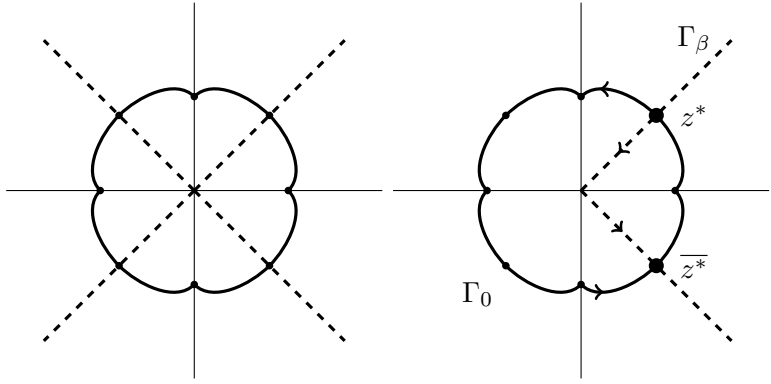

 \subsection{Proof of Proposition \ref{prop1}}
 
We need a simple lemma.
\begin{lemma} \label{angle-lemma}
Fix $k \geq 3$ and $0<\gamma<1$ and let $\omega_k=e^{2\pi i/k}$. Then there exists a constant $C=C(\gamma,k)>0$ such that
\beq \label{eq:prop1claim}
\left|\frac{1-\gamma e^{it/k}}{1-\omega^j_k \gamma e^{it/k}}\right|\leq \frac{1}{1+\frac{C\pi j(\pi j+t)}{k^2}},  \qquad t\in [-\pi,\pi], \quad  |j|\leq k/2.
\eeq 
\end{lemma}
\begin{proof}
First, note that there exists a constant $\ti{C}$ such that as long as $\theta \in [-\frac{ 5 \pi}{6}, \frac{5\pi}{6}]$
\beq \label{angle-bound}
|\sin(\theta)| \geq \ti{C} |\theta|.
\eeq
Now, 
\beq
\begin{split}
|1-&\gamma e^{it/k}|=1+\gamma ^2 -2 \gamma \cos\left( \frac{t}{k}\right),\\
|1-&\omega^j_k \gamma e^{it/k}|=1+\gamma ^2 -2 \gamma \cos \left(\frac{t+2\pi j}{k}\right),
\end{split}
\eeq 
so
\beq \label{eq:prop1claima}
\left|\frac{1-\gamma e^{it/k}}{1-\omega^j_k \gamma e^{it/k}}\right|=
\frac{1}{1+\frac{2\gamma (\cos  \frac{t}{k}-\cos \frac{t+2\pi j}{k})}{1+\gamma^2-2\gamma \cos \frac{t}{k}}}\leq \frac{1}{1+\frac{2\gamma}{(1+\gamma)^2} |\cos \frac{t}{k} -\cos\frac{t+2\pi j}{k}|}.
\eeq
Since $k \geq 3$ and $|j|\leq k/2$, it follows that $\frac{t + \pi j}{k}, \frac{\pi j}{k} \in [-\frac{5 \pi}{6}, \frac{5 \pi}{6}]$ so, writing $|\cos \frac{t}{k} -\cos\frac{t+2\pi j}{k}|=2|\sin (\frac{\pi j}{k}) \sin(\frac{t + \pi j}{k})|$, we see we may apply \eqref{angle-bound} to \eqref{eq:prop1claima} to get \eqref{eq:prop1claim}.
\end{proof}

\begin{proof}[Proof of Proposition \ref{prop1}]

First, note that the term $j=0$ and the integral over $\Gamma_0$, in \eqref{kernel3a}, together form $\bbK_{sine,ext}^{\pi/k,\gamma}$. In particular, $\bbK_1^\gamma=\bbK_{sine,ext}^{\pi,\gamma}$. Thus, we need to show that the other terms in the sum go to zero as $S \rightarrow \infty$.

First, write
\beq \label{eq:proofofsine}
\begin{split}
& \left|\bbK_k^\gamma (S+s_1,x_1;S+s_2,x_2)-\bbK_{sine,ext}^{\pi/k,\gamma} (s_1,x_1;s_2,x_2)\right|\\
&\leq C(s_1,s_2)\sum_{j=1}^{k-1}\frac{1}{2\pi} \int_{-\pi}^\pi  \left|\frac{1-\gamma e^{it/k}}{1-\omega_k^j \gamma e^{it/k}}\right|^S \ d t \equiv C(s_1,s_2) F(k,S) \\
\end{split}
\eeq
where $C(s_1,s_2)>0$ is a constant. 

For $k=2$ we get
\beq \no
\begin{split}
F(2,S)&=\int_{-\pi}^\pi  \left|\frac{1-\gamma e^{it/2}}{1+\gamma e^{it/2}}\right|^S \ d t \\
& \leq \int_{-\pi}^\pi \frac{dt}{\left(1+\frac{2\gamma}{(1+\gamma)^2} \left | 2 \cos \frac{t}{2} \right | \right)^S} \rightarrow 0
\end{split}
\eeq
as $S \rightarrow \infty$.

For $k >2$ we write
\beq \label{k>2}
\begin{split}
F(k,S) &=\sum_{j=1}^{k_1}\frac{1}{2\pi} \int_{-\pi}^\pi  \left|\frac{1-\gamma e^{it/k}}{1-\omega_k^j \gamma e^{it/k}}\right|^S \ d t
+\sum_{j=-k_2}^{-1}\frac{1}{2\pi} \int_{-\pi}^\pi  \left|\frac{1-\gamma e^{it/k}}{1-\omega_k^j \gamma e^{it/k}}\right|^S \ d t\\\
\end{split}
\eeq
where $k_1$ is the integer part of  $(k-1)/2$ and $k_2=k-1-k_1$, so both $k_1 \leq k/2$ and $k_2 \leq k/2$. By \eqref{eq:prop1claim} we have
\beq \label{est-k>2}
\frac{1}{2\pi} \int_{-\pi}^\pi  \left|\frac{1-\gamma e^{it/k}}{1-\omega_k^j \gamma e^{it/k}}\right|^S \ d t
 \leq  \frac{1}{2\pi}\int_{-\pi}^\pi  \left| \frac{1}{1+\frac{C \pi j(\pi j+t)}{k^2}}\right|^S \ d t.
\eeq
For $j\neq \pm 1$ the right-hand side converges exponentially fast to zero. The terms with $j=\pm1 $ are $\mathcal O(S^{-1})$ as $S\to \infty$. In particular,  the right-hand side of \eqref{eq:proofofsine} tends to zero as $S\to \infty$ and the statement follows.
 \end{proof}
 
 \subsection{Proof of Proposition \ref{prop2}}
 \begin{proof}[Proof of Proposition \ref{prop2}]
First note that by the mean value theorem we have
 \beq
 \begin{split}
\int_{e^{-\pi i/k}}^{e^{\pi i/k}} \frac{\prod_{r=1}^{s_2}{(1-\gamma_r z)}}{\prod_{t=1}^{s_1}{(1-\omega_k^j \gamma_t z)}} \frac{dz}{z^{x_1-x_2+1}}
=\frac{2 \pi i}{k}  
 \frac{\prod_{r=1}^{s_2}{(1-\gamma_r e^{i \theta_j})}}{\prod_{t=1}^{s_1}{(1-\omega_k^j \gamma_t e^{i \theta_j} )}} {e^{-i \theta_j(x_1-x_2)}}
\end{split}
\eeq
for some $\theta_j\in [-\pi/k,\pi/k]$. Note that $\theta_j\to 0$  as $k\to \infty$ uniformly in $j$, hence
\beq 
\begin{split}
\lim_{k\to \infty} 
\sum_{j=0}^{k-1}\omega_k^{-j x_1} \int_{e^{-\pi i/k}}^{e^{\pi i/k}} \frac{\prod_{r=1}^{s_2}{(1-\gamma_r z)}}{\prod_{t=1}^{s_1}{(1-\omega_k^j \gamma_t z)}} \frac{dz}{z^{x_1-x_2+1}}\\=
\lim_{k\to \infty}\frac{2 \pi i}{k}
\sum_{j=0}^{k-1} \omega_k^{-j x_1}  \frac{\prod_{r=1}^{s_2}{(1-\gamma_r )}}{\prod_{t=1}^{s_1}{(1-\omega_k^j \gamma_t )}}.
\end{split}
\eeq
The sum on the right-hand side is a Riemann sum and hence
\beq 
\begin{split}
\lim_{k\to \infty} 
\sum_{j=0}^{k-1}& \frac{\omega_k^{-j x_1}}{2\pi i} \int_{e^{-\pi i/k}}^{e^{\pi i/k}} \frac{\prod_{r=1}^{s_2}{(1-\gamma_r z)}}{\prod_{t=1}^{s_1}{(1-\omega_k^j \gamma_t z)}} \frac{dz}{z^{x_1-x_2+1}}\\
&=
\frac{1}{2\pi i}\int_{\partial \mathbb D} \frac{\prod_{r=1}^{s_2}{(1-\gamma_r )}}{\prod_{t=1}^{s_1}{(1-\gamma_{t} w)}} \ \frac{dw}{w^{x_1+1}}.
\end{split}
\eeq 
But then by conjugating the kernel we can rewrite this as
\beq
\begin{split}
\lim_{k\to \infty}&\, \frac{ \prod_{r=1}^{s_1}{(1-\gamma_r )}}
{\prod_{r=1}^{s_2}{(1-\gamma_r )}} \, \bbK_k^\gamma (s_1,x_1;s_2,x_2)
\\&=-{{\mathbb 1}_{s_1>s_2}  \prod_{j=s_2+1}^{s_1} (1-\gamma_j)}\frac{1}{2\pi i} \oint_{\partial \mathbb D}\, \prod_{j=s_2+1}^{s_1} (1-\gamma_jw)^{-1} \frac{dw}{w^{x_1-x_2+1}}\\
&\qquad \qquad+{\prod_{r=1}^{s_1}{(1-\gamma_r )}\frac{1}{2\pi i}\oint_{\partial \mathbb D} }\, {\prod_{t=1}^{s_1}{(1-\gamma_t w)^{-1}}} \ \frac{dw}{w^{x_1+1}}.
\end{split}
\eeq
Which means that  
\beq \label{eq:predetmanip}
\begin{split}
&\lim_{k\to \infty} \, \frac{ \prod_{r=1}^{s_1}{(1-\gamma_r )}}
{\prod_{r=1}^{s_2 }{(1-\gamma_r )}} \, \bbK_k^\gamma(s_1,x_1;s_2,x_2)
\\&\ =-{\mathbb 1}_{s_1>s_2}  \textrm{Prob}(W(s_1)=x_1\  \mid  \ W(s_2)=x_2)
+\textrm{Prob}(W(s_1)=x_1).
\end{split}
\eeq
where $W(s)$ is a random walk starting from $0$ and having geometrically distributed transition probabilities. 

Note that if we view the right-hand side of \eqref{eq:predetmanip} as the entries of the matrix, we have the sum of a lower triangular matrix with zeros on the diagonal and a rank $1$ matrix for which all columns are the same. Determinants of such matrices are easily computed.  Indeed, by subtracting the last column  from all the others the determinant does not change. The result is a matrix that is the sum of a lower triangular matrix with trivial main diagonal and a matrix for which only the last column is non-trivial. The determinant is then simply the product of the top right entry and the terms on the first subdiagonal under the main diagonal. This leads to  
\beq
\label{eq:detmanip}
\begin{split}
\lim_{k\to \infty}& \det \left( \bbK_k^\gamma(s_j,x_j;s_l,x_l) \right)_{j,l=1}^n\\ 
&= \textrm{Prob}(W(s_1)=x_1) \prod_{j=1}^{n-1} \textrm{Prob}(W(s_j)=x_j \mid W(s_{j-1}=x_{j-1})).
\end{split} \eeq
Hence, the process converges to a single discrete random walk starting from $0$ and having geometrically distributed transition probabilities.
 \end{proof}

 \subsection{Proof of Proposition \ref{prop3}}
 \begin{proof}[Proof of Proposition \ref{prop3}]

 First note that, as in \eqref{k>2}, we can rewrite \eqref{kernel3a} as
 \beq 
 \begin{split}
\bbK_k^\gamma & (s,x_1;s,x_2)
=\sum_{j=-k_1}^{k_2} \frac{\omega_k^{-j x_1}}{2\pi i} \int_{e^{-\pi i/k}}^{e^{\pi i/k}}  \left(\frac{{1-\gamma_r z}}{{1-\omega^j_k \gamma_rz}}\right)^s \frac{dz}{z^{x_1-x_2+1}}.
\end{split}
\eeq
where $k_1$ is the integer part of  $(k-1)/2$ and $k_2=k-1-k_1$. By writing $z=e^{it/k}$ we can further rewrite the kernel as 
\beq \label{eq:prop3kernelrewriting}
\begin{split}
\bbK_k^\gamma (s,x_1;s,x_2)=\sum_{j=-k_1}^{k_2} \frac{\omega^{-j x_1}_k}{k2\pi } \int_{-\pi}^{\pi} \left(\frac{1-\gamma e^{it/k}}{1-\omega^j_k \gamma e^{it/k}}\right)^{s}  e^{i t(x_2-x_1)/k} d t.\end{split}
\eeq
Now set $s=\sigma k^2$ for large $k$. In case $|j|\geq k^{1/4}$, the estimate \eqref{est-k>2} from the proof of Proposition \ref{prop1} shows that the terms are exponentially small in $k$ as $k\to \infty$. 

If $|j|\leq k^{1/4}$, we write
\beq \no
\log\left(\frac{1-\gamma e^{it/k}}{1-\gamma e^{i(t+2  \pi j )i/k}}\right)=
\log\left(1-\frac{\gamma\left(e^{it/k}-1 \right)}{1-\gamma} \right)-\log\left(1-\frac{\gamma\left(e^{i(t+2 \pi j)/k}-1 \right)}{1-\gamma} \right).
\eeq
Noting
\beq \no
1-\frac{\gamma\left(e^{it/k}-1 \right)}{1-\gamma}=
1-\frac{\gamma\left(\frac{it}{k}+\frac{(it)^2}{k^2}+\mathcal{O}(k^{-3}) \right)}{1-\gamma} 
\eeq
and similarly for $1-\frac{\gamma\left(e^{i(t+2 \pi j)/k}-1 \right)}{1-\gamma}$, and expanding the $\log$ as well, we get
\beq \label{eq:approxintegrandprop3}
 \begin{split}
 \left(\frac{1-\gamma e^{it/k}}{1-\gamma e^{i(t+ 2 \pi j)/k}}\right)^{\sigma k^2}
 =\exp\left(\frac{2 \pi   i \gamma \sigma k}{1-\gamma}j -\pi d j^2 - d j  t  \right) \left(1+\mathcal O(k^{-1/4})\right), 
\end{split}
 \eeq
with $d=\frac{2\pi \gamma \sigma}{(1-\gamma)^2}$, as $k \to \infty$. This approximation is uniform for $|j|\leq k^{1/4}$ and for $t\in [-\pi,\pi]$. 

Substituting \eqref{eq:approxintegrandprop3} in the integral gives 
\beq \label{eq:approxintegralprop3}
\begin{split}
 \int_{-\pi}^{\pi} &\left(\frac{1-\gamma e^{it/k}}{1-\omega^j_k \gamma e^{it/k}}\right)^{\sigma k^2}  e^{i t(x_2-x_1)/k} d t 
 =-\exp\left(\frac{2\pi i \gamma \sigma k}{1-\gamma} j-\pi d j^2 \right) \\
 & \times \frac{e^{- d j \pi} e^{-\pi i (x_1-x_2)/k}- e^{ d j \pi} e^{\pi i (x_1-x_2)/k}}{dj+i(x_1-x_2)/k} \left(1+\mathcal O(k^{-1/4})\right),  \qquad k\to \infty.
 \end{split}
\eeq

Now, in \eqref{eq:prop3kernelrewriting} replace the terms with $|j|\leq k^{1/4}$ by \eqref{eq:approxintegralprop3} and use the fact that the other terms are exponentially small in $k$ to obtain (recall $x_j= \textrm{integer part of  } \Big(\sigma k^2 \gamma/(1-\gamma) +k \eta_j\big)$) 
\beq 
\begin{split}
\lim_{k\to \infty} &
k\, \bbK^\gamma_k (s_1,x_1;s_2,x_2)=
 -\frac{1}{2\pi }\sum_{j=-\infty}^\infty  \exp(- 2\pi i \eta_1 j- \pi d  j^2  )\\
&\left(\frac{\exp\left(-\pi(d j + i(\eta_1-\eta_2))\right)
-\exp\left(\pi(d j + i(\eta_1-\eta_2))\right)}{dj+ i (\eta_1-\eta_2) }\right).
 \end{split}
 \eeq 
In order to show that this is \eqref{kernelJoh} we first split the sum into two separate sums 
\beq 
\begin{split}
\lim_{k\to \infty}&
k\, \bbK_k^{\gamma} (s_1,x_1;s_2,x_2)=\\
&-\frac{1}{2\pi }\sum_{j=-\infty}^\infty  \exp(- \pi d  j(j+1))
\left(\frac{\exp(- 2\pi i \eta_1 j) \exp\left(-\pi i(\eta_1-\eta_2)\right)}{dj+i (\eta_1-\eta_2)}\right)\\
&+\frac{1}{2\pi }\sum_{j=-\infty}^\infty  \exp(- \pi d j(j-1))
\left(\frac{ \exp(-  2\pi i \eta_1 j) \exp\left(\pi i(\eta_1-\eta_2)\right)}{dj+i (\eta_1-\eta_2)}\right).
 \end{split}
 \eeq 
By changing $j\mapsto -j$ in the first  sum we obtain
\beq 
\begin{split}
\lim_{k\to \infty}
& k\, \bbK_k^{\gamma} (s_1,x_1;s_2,x_2)=\\
&\frac{1}{2\pi }\sum_{j=-\infty}^\infty   \exp(- \pi d  j(j-1))
\left(\frac{\exp( 2\pi i \eta_1 j) \exp\left(-\pi i(\eta_1-\eta_2)\right)}{dj-i (\eta_1-\eta_2)}\right)\\
&+\frac{1}{2\pi }\sum_{j=-\infty}^\infty  \exp(- \pi d j(j-1))
\left(\frac{ \exp(-2\pi i \eta_1 j) \exp\left(\pi i(\eta_1-\eta_2)\right)}{dj+i (\eta_1-\eta_2)}\right)
 \end{split}
 \eeq
 and this may be written as
 \beq 
\begin{split}
\lim_{k\to \infty}
&k \bbK_k^{\gamma} (s_1,x_1;s_2,x_2)\\=
&\frac{1}{\pi} \sum_{j=-\infty}^\infty   \exp(- \pi d j(j-1))
\Re \left(\frac{ \exp\left(\pi i((2j-1)\eta_1+\eta_2)\right)}{i (\eta_2-\eta_1)+dj}\right), \end{split}
 \eeq 
 which is the statement.
 \end{proof}


\begin{thebibliography}{99}

\bibitem{BKMM}  J.~Baik, T.~Kriechenbauer, K.~T.-R.~McLaughlin and P.~D.~Miller, \emph{Discrete Orthogonal Polynomials},   Annals of Mathematics Studies, 164. Princeton University Press, Princeton, NJ, 2007. 
		
 \bibitem{Bor} A.~Borodin, \emph{Periodic Schur Process and Cylindric Partitions}, Duke Math.\ Jour.\  {\bf 10}  (2007), no.\ 4, 1119--1178.
  
 \bibitem{BorDet} A.~Borodin, \emph{Determinantal point processes},  In: Oxford Handbook on Random Matrix theory, edited by Akemann G.; Baik, J. ; Di Francesco P., Oxford University Press, 2011. (arXiv:0911.1153) 
 
\bibitem{BGR} A.~Borodin, V.~Gorin and E.~Rains, \emph{$q$-distributions on boxed plane partitions},   Selecta Math.\ (N.S.) {\bf 16} (2010), no.\ 4, 731--789.



\bibitem{BH}  E.~Br\'ezin and S.~Hikami, \emph{Correlations of nearby levels induced by a random potential}, Nuclear Phys.\ B {\bf 479} (1996), no.\ 3, 697--706. 

\bibitem{dyson} F.~Dyson, \emph{A Brownian-motion model for the eigenvalues of a random matrix}, J.\ Math.\ Phys.\ {\bf 3} (1962), 1191--1198.

\bibitem{Erd} L.~Erd\H os, B.~Schlein, H.-T.~Yau and J.~Yin, \emph{The local relaxation flow approach to universality of the local statistics for random matrices}, to appear in Annales Inst.\ H.\ Poincar\'e (B), Probability and Statistics, preprint arXiv:0911.3687.

\bibitem{Forrester} P.~Forrester, \emph{Log-gases and Random Matrices}, London Mathematical Society Monographs Series, 34. Princeton University Press, Princeton, NJ, 2010. 

\bibitem{grabiner} D.~J.~Grabiner,  \emph{Brownian motion in a Weyl chamber, non-colliding particles, and random matrices},
Ann. Inst. H. Poincar\'e Probab.\ Statist.\ {\bf 35} (1999), no.\ 2, 177--204. 

\bibitem{HKPV} J.~B.~Hough, M.~Krishnapur, Y.~Peres and B.~Vir\'ag, \emph{Determinantal processes and independence}, Prob.\ Surv.\ {\bf 3} (2006), 206--229.

\bibitem{ImSa} T. Imamura and T. Sasamoto, \emph{Correlation function of the Schur process with a fixed final partition}, J. Math. Phys. {\bf 49} (2008), no.\ 5, 053302, 20 pp.

\bibitem{J4} K.~Johansson, \emph{Universality of the local spacing distribution in certain ensembles of Hermitian Wigner Matrices}, Commun.\ Math.\ Phys.\ {\bf 215} (2001), 683--705.

\bibitem{J2} K.~Johansson, \emph{Non-intersecting paths, random tilings and random matrices},  Probab.\ Theory Related Fields {\bf 123} (2002), no.\ 2, 225--280.

\bibitem{J3} K.~Johansson, \emph{Determinantal processes with number variance saturation},  2001
Commun.\ Math.\ Phys.\ {\bf 252} (2004), no. 1-3, 111--148. 

\bibitem{J} K.~Johansson, \emph{Random matrices and determinantal processes},  Mathematical Statistical Physics, Elsevier B.V.\ Amsterdam (2006) 1--55.

\bibitem{K} W.~K\"onig, \emph{Orthogonal polynomial ensembles in probability theory}, Probab.\ Surveys {\bf 2} (2005), 385--447.

\bibitem{L} R.~Lyons, \emph{Determinantal probability measures}, Publ.\ Math.\ Inst.\ Hautes Etudes Sci.\ {\bf 98}  (2003), 167--212.

\bibitem{Ok1} A.~Okounkov, \emph{Infinite wedge and random partitions}, Selecta Math.\ (N.S.) {\bf 7} (2001),
57--81.

\bibitem{O} A.~Okounkov, \emph{Symmetric functions and random partitions}, Symmetric functions 2001: surveys of developments and perspectives, 223--252, NATO Sci.\ Ser.\ II Math.\ Phys.\ Chem., 74, Kluwer Acad.\ Publ., Dordrecht, 2002.

\bibitem{OR} A.~Okounkov and N.~Reshetikhin,   \emph{Correlation function of Schur process with application to local geometry of a random 3-dimensional Young diagram.} J.\ Amer.\ Math.\ Soc.\ {\bf 16} (2003), no.\ 3, 581--603 (electronic).

\bibitem{Sosh} A.~Soshnikov, \emph{Determinantal random point fields},  Uspekhi Mat.\ Nauk {\bf 55} (2000), no.\ 5 (335), 107--160; translation in Russian Math.\ Surveys {\bf 55} (2000), no.\ 5, 923--975.

\bibitem{Sosh2} A. Soshnikov, \emph{Determinantal random point fields,}  in: Encyclopedia of Mathematical Physics, 47--53. Oxford: Elsevier, 2006.

\bibitem{Stem} J. R. Stembridge, \emph{Nonintersecting Paths, Pfaffians, and plane partitions,} Adv.\ in Math.,  {\bf 83} (1990), 96--131.
\end{thebibliography}
\end{document}